\documentclass[12pt,a4paper,reqno]{amsart}
\usepackage{amsfonts}
\usepackage{amsthm}
\usepackage{amsmath}
\usepackage{amscd}
\usepackage[latin2]{inputenc}
\usepackage{t1enc}
\usepackage[mathscr]{eucal}
\usepackage{indentfirst}
\usepackage{graphicx}
\usepackage{graphics}
\usepackage{tikz}
\usepackage{pict2e}
\usepackage{epic}
\numberwithin{equation}{section}
\usepackage[margin=2.9cm]{geometry}
\usepackage{epstopdf} 
\usepackage{tabularx}

\theoremstyle{plain}
\newtheorem{Th}{Theorem}[section]
\newtheorem{Lemma}[Th]{Lemma}
\newtheorem{Cor}[Th]{Corollary}
\newtheorem{Prop}[Th]{Proposition}

 \theoremstyle{definition}
\newtheorem{Def}[Th]{Definition}

\newtheorem{Rem}[Th]{Remark}
\newtheorem{?}[Th]{Problem}
\newtheorem{Ex}[Th]{Example}

\begin{document}

\title{A quotient of Fomin-Kirillov algebra and $Q$-Lucas polynomial}

\author[S. Homayouni]{Sirous Homayouni}

\address{Sirous Homayouni,  Department of Mathematics and Statistics, York University, Toronto, Canada}
\email{chomayou@mathstat.yorku.ca, cyhomayou@gmail.com}

 \subjclass[2010]{Primary: Mathematics. Secondary: Algebraic Combinatorics}


\begin{abstract}We 
	introduce a quotient of Fomin-Kirillov algebra $FK(n)$ denoted $\overline{FK}_{C_n}(n)$, over the ideal generated by the edges of a complete graph on $n$ vertexes that are missing in the $n$-cycle graph $C_n$. 
	For this quotient algebra $\overline{FK}_{C_n}(n)$,	
	we show that the basis is in one-to-one correspondence with the set of matchings in an $n$-cycle graph. We also prove that the dimension of $\overline{FK}_{C_n}(n)$ equals the Lucas Number $L_n$ and its Hilbert series is $q$-Lucas polynomial.	
   We find the character map of this quotient algebra over Dihedral group $D_n$. 
\end{abstract}
\begin{raggedright}
\maketitle
\section{Motivation and Preliminaries}
Fomin-Kirillov algebra $FK(n)$ is a quadratic non-commutative algebra on generators $x_{ij}=-x_{ji}$, where $1\leq i<j\leq n$, 
that satisfy the following relations \cite{fomin}.
\begin{equation}\label{rel}
\begin{split}
(i) : x_{ij}^2=0,~1\leq i < j \leq n;\\
(ii) : x_{ij}x_{kl} - x_{kl}x_{ij} = 0;~\text{for distinct}~ i, j, k, l~ \text{such that}~1\leq i, j, k<l \leq n;\\
(iii) : x_{ij}x_{jk} - x_{jk}x_{ik} - x_{ik}x_{ij} = 0; ~\text{if}~ 1\leq i < j < k\leq n;\\
(iii^\prime) : x_{ij}x_{ik} - x_{jk}x_{ij} + x_{ik}x_{jk} = 0; ~\text{if}~ 1\leq i < j < k\leq n.
\end{split}
\end{equation}
If $F\langle x_{ij}\rangle$ is the free associated algebra generated by $x_{ij},~1\leq i<j\leq n$, then $FK(n)$ is the quotient of $F\langle x_{ij}\rangle$ over the ideal generated by the relations in \eqref{rel}. 

While the motivation of Fomin and Kirillov for introducing this algebra was calculation of the structure constant of Schubert polynomials, however this algebra received later much attention in algebraic and combinatorial aspects \cite{CHRISTOPH}, \cite{Karola}, \cite{liu}, \cite{fomin''}, \cite{bazlov'},
\cite{kirillov}, \cite{kirillov2}.\\
The analysis of the sub-algebras of $FK(n)$, made a better understanding of this algebra by showing that a striking amount of their structure parallels Coxeter groups and nil-Coxeter algebras \cite{Karola}.

Despite many research works on different aspects of this algebra, still there are many questions about this algebra yet to be answered including the question of dimensionality; while the dimension of $FK(n)$ is known to be finite for $n=3, 4, 5$, however it is not even known if it is finite or not for $n\geq 6$, \cite{CHRISTOPH}, \cite{Karola}.

In this paper, to better understand the structure of $FK(n)$ we study a quotient of it denoted  $\overline{FK}_{C_n}(n)$ associated with the subgraph $n$-cycle $C_n$ of the complete graph on $n$ vertexes; while $FK(n)$ is associated to the complete graph on $n$ vertexes, $\overline{FK}_{C_n}(n)$ is associated with the subgraph $n$-cycle $C_n$  of the complete graph on $n$ vertexes. We find a beautiful connection between this quotient algebra and the theory of Lucas polynomials. Specifically we find that the Hilbert series of this quotient algebra is the $q$-Lucas polynomial, and its dimension equals the Lucas number, As well the symmetry properties of this quotient algebra under $D_n$ is related to $q$-Fibonacci polynomials.
\subsection{Gr\"obner basis} For the purposes of this paper, the generating set for the defining ideal of the algebra is Gr\"obner basis. 
 From non-commutative Buchberger theory \cite{mora}, the reduced Gr\"obner basis is unique up to monomial ordering. 
 
 The monomial ordering
we use in this work is called graded lexicographic ordering (glex) defined as follows.
\begin{Def}\label{d1}	For monomials $M_1$ and $M_2$ we say 
	\begin{displaymath}
	M_1<_{glex}M_2~ \text{if}~\begin{cases}
	deg M_1<deg M_2, ~\text{or}\\
	deg M_1=deg M_2~\text{and}~M_1<_{lex}M_2,
	\end{cases}
	\end{displaymath}
	where 
	the lexicographic ordering (lex) of monomials that we use in this work
	is defined by first introducing a variable ordering by
	\begin{displaymath}
	x_{ij} > x_{kl}~ \text{if}~
	\begin{cases}
	j < l, ~\text{or}\\
	j = l ~\text{and}~ i > k,\end{cases}
	\end{displaymath}
	then with this variable ordering, the following rule completes the
	definition of our lexicographic monomial ordering. For monomials $M_1$ and $M_2$ of the same usual degree $d$,~ $M_1 <_{lex} M_2$ if the first variable of $M_1$ is less than the first variable of $M_2$.
	if the first $k$ variables happen to be the same, then compare the $k+1$ st variables.
\end{Def}
\section{A quotient algebra of $FK(n)$}\label{qalg}
We introduce the algebra $\overline{FK}_{C_n}(n)$ as the quotient of $FK(n)$ over the ideal generated by the edges of the complete graph on $n$ vertexes that are missing in subgraph $n$-cycle $C_n$.
\begin{displaymath}\begin{split}
\overline{FK}_{C_n}(n)=&
\frac{FK(n)}{I\langle\text{missing edges in subgraph}~ C_n\rangle}
\\=&
\frac{F\langle x_{ij}\rangle}{I\langle \{\text{generators of the defining ideal of $FK(n)$}\}\cup \{\text{missing edges in}~C_ n\}\rangle}.
\end{split}
\end{displaymath}
In other words by putting the missing edges in \eqref{rel} equal to zero we come up with the new set of relations \eqref{2terms} associated to $\overline{FK}_{C_n}(n)$ as in the following definition. 
\begin{Def}\label{def-cnalg}
	For $n>3$, we define $\overline{FK}_{C_n}(n)$, the quotient algebra of $FK(n)$, as the algebra on generators $\{x_{1n}=-x_{n1},~x_{m,m+1}=-x_{m+1,m}, 1\leq m\leq  n-1\}$, 
	that satisfies the following relations (Leading terms are underlined).
	\begin{equation}\label{2terms}
	\begin{split}
	R_{C_n}=&\{ x_{m,m+1}^2=0,~1\leq m\leq n-1,~
	x_{1,n}^2=0,
	\\&
	x_{m,m+1}x_{m+1, m+2}=0,
	~x_{m+1,m+2}x_{m,m+1}=0,~ 1\leq m \leq n-2,\\&
	x_{n-1,n}x_{1n}=0,~x_{1n}x_{n-1,n}=0,~x_{12}x_{1n}=0,~x_{1n}x_{12}=0,\\&\underline{x_{m,m+1}x_{l,l+1}}-x_{l,l+1}x_{m,m+1}=0, ~1\leq m\leq l-2,~ 3\leq l\leq n-1,\\&
	\underline{x_{m,m+1}x_{1,n}}-x_{1n}x_{m,m+1}=0,~ 2\leq m\leq n-2\}.
	\end{split}
	\end{equation}
\end{Def}	


From the relations in \eqref{2terms} we have the following set of generators for the defining ideal of $\overline{FK}_{C_n}(n)$.
\begin{equation}\label{2terms''}
\begin{split}
&\text{Generators of defining ideal of}~ \overline{FK}_{C_n}(n)=\\&\{ x_{m,m+1}^2,~1\leq m\leq n-1,~
x_{1,n}^2,
\\&
x_{m,m+1}x_{m+1, m+2},
~x_{m+1,m+2}x_{m,m+1},~ 1\leq m \leq n-2,\\&
x_{n-1,n}x_{1n},~x_{1n}x_{n-1,n},~x_{12}x_{1n},~x_{1n}x_{12},\\&\underline{x_{m,m+1}x_{l,l+1}}-x_{l,l+1}x_{m,m+1}, ~1\leq m\leq l-2,~ 3\leq l\leq n-1,\\&
\underline{x_{m,m+1}x_{1,n}}-x_{1n}x_{m,m+1},~ 2\leq m\leq n-2\}.
\end{split}
\end{equation}
\begin{Rem}
	As mentioned above while the generators of $FK(n)$ are the edges of a complete graph on $n$ vertexes, the generators of $\overline{FK}_{C_n}(n)$ are the edges of the subgraph $C_n$ of the complete graph on $n$ vertexes, i.e., $\{x_{12}, x_{23},\cdots, x_{n-1,n},~ x_{1n}\}$. Therefore while $S_n$ defines an action on $FK(n)$, as the defining ideal of $FK(n)$ is stable under $S_n$, however $S_n$ does not do so on $\overline{FK}_{C_n}(n)$, as the defining ideal of $\overline{FK}_{C_n}(n)$ generated by the terms in \eqref{2terms''} is not stable under $S_n$ (for example $x_{1n}^2$ in \eqref{2terms''}, under transposition $(12)\in S_n$ goes to $x_{2n}^2$ which is not in the set). 
	
	However, as we will see later, the ideal is stable under Dihedral group $D_n$, so $D_n$ defines an action on $\overline{FK}_{C_n}(n)$ (an action is defined on a quotient if and only if the defining ideal is stable under the action).\end{Rem}
\subsection{Dihedral group} Dihedral group is defined by
\begin{displaymath}
D_n=\langle r, s|s^2=1, r^n=1, (rs)^2=1\rangle.
\end{displaymath} 
We realize group $D_n$ in the permutation group $S_n$, where
\begin{equation}\label{r,s}\begin{split}
 r=(12\cdots n)~\text{and}~s&=\begin{cases}
(1, n)(2, n-1)\cdots (\frac{n}{2}, \frac{n}{2}+1), & \text{even}~n,\\
(1, n)(2, n-1)\cdots (\frac{n-1}{2}, \frac{n+3}{2})(\frac{n+1}{2}), & \text{odd}~n, 
\end{cases}
\end{split}\end{equation}  
are permutations on $n$ vertexes (rotation and reflection in $C_n$, see Figure \ref{rs}).
	\begin{figure}
				\begin{tikzpicture}
		\path[draw] (0,0) node[left] (6) {6}
		-- (-0.7,-1) node[below] (5) {5}
		-- (0, -2) node[below] (4) {4}
		-- (1.5,-2) node[right] (3) {3}
		-- (2.1, -1) node[above] (2) {2}
		-- (1.4,0) node[above] (1) {1}
		-- cycle;
		\path[draw] (.75,0) node[left] () {}
		-- (.75,-2) node[right] () {}
		-- cycle;~~~~
		\end{tikzpicture}
		\begin{tikzpicture}
		\path[draw] (-0.08,0) node[left] (5) {5}
		-- (-.5,-1.2) node[below] (4) {4}
		-- (0.7, -2) node[below] (3) {3}
		-- (2,-1.2) node[right] (2) {2}
		-- (1.5, 0) node[above] (1) {1}
		-- cycle;
		\path[draw] (.7,0) node[below] () {}
		-- (.7,-2) node[below] () {}
		-- cycle;
		\end{tikzpicture}
		\caption{Examples of reflections $s$ in an $n$-cycle for even/odd $n$:
			$s=(16)(25)(34)$ in a $6$-cycle and $s=(15)(24)(3)$ in a $5$-cycle.} 
		\label{rs}
	\end{figure}
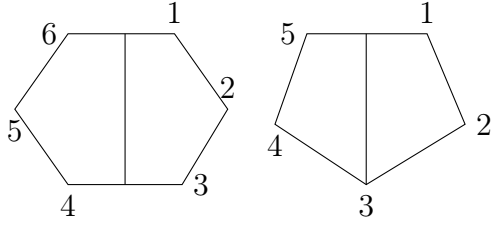
\subsection{The stability of the defining ideal of $\overline{FK}_{C_n}(n)$ under $D_n$}
 It is easily shown that the set of generators of the defining ideal of $\overline{FK}_{C_n}(n)$ in \eqref{2terms''} is invariant under $D_n$. So the defining ideal of $\overline{FK}_{C_n}(n)$ is stable under $D_n$. Therefore $D_n$ defines an action on $\overline{FK}_{C_n}(n)$, as it follows.

\subsection{The action of $D_n$ on $\overline{FK}_{C_n}(n)$}We define the action of $D_n$ realized in permutation group on $\overline{FK}_{C_n}(n)$ as \begin{equation}\label{d_n-def}
D_n:\overline{FK}_{C_n}(n)\to \overline{FK}_{C_n}(n) ~\text{by}~ \sigma(M+I)\mapsto \sigma M+I,
\end{equation}where the action of $\sigma\in D_n$ on a monomial $M$ is defined via the action on a variable $x_{ij},~j=i+1$ (as $x_{ij}$ represents an edge of an $n$-cycle), 
defined by 
\begin{equation}\label{sigmax}
\sigma x_{ij}=\begin{cases}
x_{\sigma(i)\sigma(j)}~&\text{if}~ \sigma(i)<\sigma(j),\\
-x_{\sigma(j)\sigma(i)}~&\text{otherwise},
\end{cases}
\end{equation}
and multiplicative extension of it to monomial $M$. 

The algebra is closed under $D_n$ as the defining ideal is stable under $D_n$. Regarding the action defined in \eqref{d_n-def}, we need to check the following items.
\begin{enumerate}
	\item Well defined: 
	\begin{displaymath}\begin{split}
	&M+I= M'+I\to M-M'\in I\to \sigma(M-M')\in I \text{ (as $I$ is stable under $D_n$)}.\\ &\text{Then}~\sigma M-\sigma M'\in I \to \sigma M+I=\sigma M'+I\to \sigma(M+I)=\sigma(M'+I).
	\end{split}\end{displaymath}
	\item Identity axiom: $\epsilon(f(x_{i_1j_1},\cdots,x_{i_kj_k})+I)=f(x_{i_1j_1},\cdots,x_{i_kj_k})+I$,
	\item Compatibility axiom:
	\begin{displaymath}\begin{split}
	(\sigma_1\sigma_2)(f(x_{i_1j_1},\cdots,x_{i_kj_k})+I)=&f((\sigma_1\sigma_2)x_{i_1j_1},\cdots, (\sigma_1\sigma_2)x_{i_kj_k})+I\\=&f(x_{(\sigma_1\sigma_2)i_1(\sigma_1\sigma_2)j_1},\cdots, x_{(\sigma_1\sigma_2)i_k(\sigma_1\sigma_2)j_k})+I\\=&f(x_{\sigma_1(\sigma_2i_1)\sigma_1(\sigma_2j_1)},\cdots, x_{\sigma_1(\sigma_2i_k)\sigma_1(\sigma_2j_k})+I\\=&\sigma_1f(x_{\sigma_2i_1\sigma_2j_1},\cdots, x_{\sigma_2i_k\sigma_2j_k})+I\\=
	&
	\sigma_1(\sigma_2f(x_{i_1j_1},\cdots, x_{i_kj_k}))+I.
	\end{split}\end{displaymath}\end{enumerate}	
Hence the map in \eqref{d_n-def}, defines an action on $\overline{FK}_{C_n}(n).$
\subsection{Representation decompositions of $\overline{FK}_{C_n}(n)$}\begin{itemize}
	\item (Usual degree decomposition) The action of $D_n$ on $\overline{FK}_{C_n}(n)$ defined in \eqref{d_n-def} permutes the indexes among themselves but does not change the number of variables in a monomial. Therefore usual degree remains invariant under $D_n$, so usual degree representation decomposition makes sense for $\overline{FK}_{C_n}(n)$ (i.e., $D_n$ respects decomposition in usual degree).
	\begin{equation}\label{oplus1}
	\overline{FK}_{C_n}(n)=\bigoplus_{d\geq 0} \overline{FK}_{C_n}^{(d)}(n). 
	\end{equation}
	\item (Conjugacy class decomposition) $S_n$-degree is well defined on $\overline{FK}_{C_n}(n)$, since the generators of the defining ideal in \eqref{2terms''} are homogeneous with respect to $S_n$. The action of $D_n$ on $\overline{FK}_{C_n}(n)$ defined in \eqref{d_n-def} does not change the cycle type of the $S_n$-degrees assigned to elements of $\overline{FK}_{C_n}(n)$, because a monomial of degree $\sigma$, when acted upon by a permutation $\nu$ is sent to a monomial of degree $\nu^{-1}\sigma\nu$ which is a conjugate of $\sigma$ by definition of conjugacy. So permutation $\nu$ sends $\sigma$ to a conjugate of $\sigma$, i.e., conjugacy class is invariant under $D_n$ realized in permutation group. Hence the conjugacy class decomposition is a representation decomposition.
	\begin{equation}\label{oplus}
	\overline{FK}_{C_n}(n)=\bigoplus_{\mu} \overline{FK}_{C_n}^{\mu}(n), 
	\end{equation}where $\overline{FK}_{C_n}^{\mu}(n)$ stands for all the elements of $S_n$-degree $\sigma$ in $\overline{FK}_{C_n}(n)$ that belong to conjugacy class indexed by $\mu$.
	\begin{Rem}\label{s_n-deg'}
		Every permutation $\sigma$ can be obtained using only transpositions $(1,2)$,~$(2,3),\cdots (n-1,n)$, so $\overline{FK}_{C_n}^{\sigma}(n)$ is potentially not zero for any permutation $\sigma$. However for our special cases we will see in sections \ref{chareven} and \ref{charodd} that for many of permutations it is zero.
		
		Having an $S_n$-conjugacy class, it is not true in general that it is a $D_n$-conjugacy class as well. For example transpositions $(2,3)$ and $(2,5)$ in the same $S_n$-conjugacy class for $n>5$, are not conjugated via an element in $D_n$.  However we will see in section \ref{b'} from our very special basis \eqref{b} that any two element in an $S_n$ conjugacy class can be conjugated via a rotation or reflection in $D_n$ realized in permutation group. I.e., an $S_n$ conjugacy class is a $D_n$ conjugacy class as well.
		 It is why \eqref{oplus} can not be refined. 
		
	\end{Rem}
	\item (Set partition type decomposition) Set-partition degree is well defined on $\overline{FK}_{C_n}(n)$, since the generators of the defining ideal in \eqref{2terms''} are homogeneous with respect to set-partition degree.	
	For any monomial $M$ in $\overline{FK}_{C_n}(n)$, the appearance of $x_{ij}$ in $M$ implies that $i,j$ are in the same part of the set-partition degree. Any $x_{ij}$ appearing in $M$, under the action of $\sigma \in D_n$ (by definition in \eqref{sigmax}) goes to $x_{\sigma(i)\sigma(j)}$ if $\sigma(i)<\sigma(j)$ ~and to ~$-x_{\sigma(j)\sigma(i)}$~otherwise, so as a result of the action, now we have the indexes $\sigma(i)$ and $\sigma(j)$ in the same part. This means that while the action of $D_n$ changes the indexes in a part but does not change the number of them in that part, i.e., set partition type is invariant under the action of $D_n$ realized in permutation group. Hence partition in terms of set partition type is a $D_n$ representation decomposition. 
	\begin{equation}\label{oplus1}
	\overline{FK}_{C_n}(n)=\bigoplus_{\phi} \overline{FK}_{C_n}^{\phi}(n), 
	\end{equation}
	where $\overline{FK}_{C_n}^{\phi}(n)$ stands for all set partition degree $\lambda$ elements that belong to the same set partition type indexed by $\phi$. 
\end{itemize}
\begin{Rem}
	Unlike permutations, discussed in Remark \ref{s_n-deg'}, not every set-partition can be obtained by joining only consecutive entries modulo $n$. For example it is not possible to obtain the set-partition $\{\{1,3\}, \{2\}, \{4\}\}$. Therefore for some set-partition $A$, the space $\overline{FK}_{C_n}^{A}(n)$ is zero. We will see later in section \ref{b'} from our very special basis \eqref{b} that for any two non-zero  $\overline{FK}_{C_n}^{A}(n)$ and $\overline{FK}_{C_n}^{B}(n)$ such that $A$ and $B$ have same partition type, we can find an element of $D_n$ that maps $A$ into $B$. Therefore a class of the same set partition type can be a $D_n$ conjugacy class as well. Hence  \eqref{oplus1} can not be refined.
\end{Rem}
\subsection{Gr\"obner basis}
 From the non-commutative Buchberger theory \cite{mora}, the set of generators of the defining ideal of $\overline{FK}_{C_n}(n)$ in \eqref{2terms'''} forms the degree $2$ elements of the Gr\"obner basis. To find the degree $3$ terms, we need to reduce the $S$-polynomials between any two pair of the degree $2$ terms with respect to lower degree terms.
 
We notice that the $S$-polynomial between two monomials is always zero. Moreover, the only $S$-polynomials that reduce to non-zero elements for Gr\"obner basis are \begin{displaymath}
S(\underline{x_{m,m+1}x_{1n}}-x_{1n}x_{m,m+1}, x_{1n}x_{12}, 3), ~\text{and}~ S(x_{1,n}x_{12}, x_{m,m+1}x_{l,l+1}-x_{l,l+1}x_{m,m+1}, 3).
\end{displaymath}
The first of the above however, reduces to
\begin{displaymath}
A=\{ x_{1n}x_{m,m+1}x_{12},~ 3\leq m\leq n-2\}.
\end{displaymath} The range of $m$ refers to the fact that $m$ has to be more than $2$ otherwise it is divisible by the relation $x_{m+1,m+2}x_{m,m+1}=0,~ 1\leq m \leq n-2$ for $m=1$.
Also the second one reduces to  
\begin{displaymath}
B=\{x_{1n}x_{l,l+1}x_{12},~ 3\leq l\leq n-2\}.
\end{displaymath} Here the range refers to the fact that $l$ has to be at most $n-2$ for $x_{1n}x_{l,l+1}x_{12}$ not to be divisible by relation $x_{1n}x_{n-1,n}=0$. 

However the sets $A$ and $B$ are the same. Since the set of elements of degree $3$ in the set of relations of $\overline{FK}_{C_n}(n)$ is empty, the elements $x_{1n}x_{l,l+1}x_{12},3\leq m\leq n-2$ are the whole elements of Gr\"obner basis of degree $3$. Therefore we have
 \begin{equation}\label{3terms}
\text{degree 3 terms}: \{x_{1n}x_{m,m+1}x_{12}: ~ 3\leq m\leq n-2\}.
\end{equation} 
\begin{Ex}
	For $n=4$, we have no degree $3$ terms. For $n=5$, the only degree $3$ term of the Gr\"obner basis is  $x_{15}x_{34}x_{12}$. For $n=6$, we have two degree $3$ terms: $x_{16}x_{34}x_{12}$ and $x_{16}x_{45}x_{12}$.
\end{Ex}
To find the degree $4$ terms, the only $S$-polynomial is (as the rest of them goes zero)
\begin{displaymath}
S(\underline{x_{m_2,m_2+1}x_{1n}}-x_{1n}x_{m_2,m_2+1},x_{1n}x_{m_1,m_1+1}x_{12},4),
\end{displaymath}  
 which reduce to
 \begin{displaymath}
 x_{1n}x_{m_2,m_2+1}x_{m_1,m_1+1}x_{12},
 \end{displaymath}  with respect to terms of degree less than $4$. However by the relations in Equation \eqref{2terms} these terms goes zero unless $ m_2-m_1\geq 2$, $m_1\geq 3$ and $m_2\leq n-2$. 
 These are the only degree $4$ terms as the set of elements of degree $4$ in the set of relations of $\overline{FK}_{C_n}(n)$ is empty. 
 Therefore we have
\begin{equation}\label{4terms}\text{degree $4$ terms}: \{x_{1n}x_{m_2,m_2+1}x_{m_1,m_1+1}x_{12}:~ m_2-m_1\geq 2,~m_1\geq 3 ~\text{and} ~m_2\leq n-2\}.
\end{equation} 
The forms of the degrees $3$ and $4$ terms in Equations \eqref{3terms} and \eqref{4terms} suggest a general form for the elements of the basis of $\overline{FK}_{C_n}(n)$ of degree $k$, in the following proposition.
\begin{Prop}\label{gbdegk}
	For $k\geq 4$ any degree $k$ element of Gr\"obner basis for the ideal associated to $\overline{FK}_{C_n}(n)$ is of the form 	
	\begin{equation}\label{kterms}
	\begin{split}
	&x_{1n}x_{m_{k-2}m_{k-2}+1}x_{m_{k-3}m_{k-3}+1} \cdots x_{m_2m_2+1}x_{m_1m_1+1}x_{12},~ \text{where},\\& m_{i}-m_{i-1}\geq 2,~i=2, 3, \cdots, k-2, ~\text{and}~ ~m_1\geq 3 ~\text{and} ~m_{k-2}\leq n-2.
	\end{split}
	\end{equation}
\end{Prop}

\begin{proof}
	Base of induction: Equation \eqref{4terms} serves as the base of induction.\\  
	Inductive step: Let the statement be valid for degree $k$ elements of the Gr\"obner basis. We need to show that it is valid for degree $k+1$, i.e., we need to show that for degree $k+1$ we will have 
	\begin{equation}\label{kterms''}
	\begin{split}
	x_{1n}x_{m_{k-1}m_{k-1}+1}x_{m_{k-2}m_{k-2}+1}x_{m_{k-3}m_{k-3}+1} \cdots x_{m_2m_2+1}x_{m_1m_1+1}x_{12},~\text{where}\\m_i-m_{i-1}\geq 2,~i=2, 3, \cdots, k-1, ~\text{and}~ ~m_1\geq 3 ~\text{and} ~m_{k-1}\leq n-2.
	\end{split}
	\end{equation}
	The only $S$-polynomials available between the degree $k$ element in the statement and a lower degree term that reduce to degree $k+1$ elements of the Gr\"obner basis is (as the rest of them go zero) \begin{equation}\label{1}\begin{split}S(x_{m_{k-1}m_{k-1}+1}x_{1n}-x_{1n}x_{m, m+1},~  x_{1n}x_{m_{k-2}m_{k-2}+1}x_{m_{k-3}m_{k-3}+1} \cdots x_{m_1m_1+1}x_{12},~ k+1)\\=x_{1n}\underline{x_{m_{k-1}m_{k-1}+1}x_{m_{k-2}m_{k-2}+1}} \cdots x_{m_1m_1+1}x_{12}.
	\end{split}\end{equation} 
	We now focus on the underlined product on the right hand side of \eqref{1}:
	 \begin{displaymath}
	C=\underline{x_{m_{k-1},m_{k-1}+1}x_{m_{k-2},m_{k-2}+1}}.
	\end{displaymath} For $C$ not to vanish or reduce to something else on reduction w.r.t. any of the degree $2$ terms in Equation \eqref{2terms}, we need to avoid the following cases.
	\begin{enumerate}
		\item $m_{k-1}=m_{k-2}$ where $C\to 0$ by relation $x_{m,m+1}^2=0$.
		\item  $m_{k-1}=m_{k-2}-1$ where $C\to 0$ by relation $x_{m,m+1}x_{m+1,m+2}=0$.
		\item $m_{k-1}\leq m_{k-2}-2$ where $C\to x_{m_{k-2}m_{k-2}+1}x_{m_{k-1}m_{k-1}+1}$. (by reduction w.r.t. the leading term of the relation $\underline{x_{m,m+1}x_{l,l+1}}-x_{l,l+1}x_{m,m+1}=0)$
		\item $m_{k-1}\leq m_{k-2}+1$ where $C\to 0  $ by relation $x_{m+1,m+2}x_{m,m+1}=0$.
	\end{enumerate}
	To avoid the above cases we need that inequality $n-2\geq m_{k-1}>m_{k-2}+1$ holds, which is nothing but $m_{k-1}-m_{k-2}\geq 2$ and $n-2\geq m_{k-1}$, as well we need $m_1\geq 3$, for $x_{m_1m_1+1}x_{12}$ not to go zero by $x_{m+1,m+2}x_{m,m+1}=0, ~1\leq m \leq n-2$. 
	
	However these three inequalities
	\begin{displaymath}
 m_{k-1}-m_{k-2}\geq 2,~ n-2\geq m_{k-1},~ \text{and}~m_1\geq 3,
 \end{displaymath}  
	meet the requirements of Equation \eqref{kterms''} for validity for degree $k+1$. This completes the proof.
\end{proof}	 
Putting together degree $k$ terms for $k\geq 4$ in Equation \eqref{kterms}, degree $3$ term of Equation \eqref{3terms} and degree $2$ terms in Equation \eqref{2terms} we get the set of Gr\"obner basis for the ideal associated with $\overline{FK}_{C_n}(n)$:
\begin{equation}\label{gb'}
\begin{split}
GB_{C_n}=&\{x_{m,m+1}^2:~1\leq m\leq n-1; ~x_{1,n}^2,\\&
x_{m,m+1}x_{m+1, m+2},
~x_{m+1,m+2}x_{m,m+1}, ~1\leq m \leq n-2\\&x_{n-1,n}x_{1n}, ~x_{1n}x_{n-1,n},~x_{12}x_{1n},~x_{1n}x_{12}\\&x_{m,m+1}x_{l,l+1}-x_{l,l+1}x_{m,m+1}, ~1\leq m\leq l-2,~ 3\leq l\leq n-1\\&
x_{m,m+1}x_{1,n}-x_{1n}x_{m,m+1},~ 2\leq m\leq n-2,\\&
x_{1n}x_{m,m+1}x_{12}, ~ 3\leq m\leq n-2,~\text{and}\\&
x_{1n}x_{m_{k-2}m_{k-2}+1}x_{m_{k-3}m_{k-3}+1} \cdots x_{m_2m_2+1}x_{m_1m_1+1}x_{12},\\&\text{where}~ m_i-m_{i-1}\geq 2,~i=2, 3, \cdots, k-1, ~\text{and}~ ~m_1\geq 3 ~\text{and} ~m_{k-1}\leq n-2\}.
\end{split}
\end{equation}
\begin{Ex}
	For $n=3$ we have
	\begin{displaymath}
	GB=\{x_{12}^2,x_{23}^2,x_{13}^2, x_{12}x_{23}, x_{23}x_{13},
	x_{12}x_{13}, x_{23}x_{12}, x_{13}x_{23}, x_{13}x_{12}\}.
	\end{displaymath}		 
\end{Ex}
\subsection{Basis and dimension}\label{b'}
\begin{Def}\label{matching''}\begin{enumerate}
\item A matching in a graph is a subset of the set of edges in the graph wherein no two edges have a common vertex.
\item A matching on a set of line segments $V_n=\{\overline{12},\overline{23},\cdots \overline{n-1,n}\}$ is a subset of $V_n$ wherein no two line segments have a common vertex. 
\end{enumerate}\end{Def} 
\begin{Ex}
	In an square with successive sides/edges $a, b, c, d$, we have $7$ sets of sides/edges where no two of them have a common vertex:
	\begin{displaymath}
	 \phi, \{a\}, \{b\}, \{c\}, \{d\}, \{a,c\}, \{b,d\}.
	 \end{displaymath} So the number of matchings in a square in $7$.\end{Ex}
\begin{center}	
	\begin{tikzpicture}
	\draw (4cm,1em) -- (5.5cm,1em) node[midway,above] {a};
	\draw (5.5cm,1em) -- (5.5cm,-2em) node[midway,right] {b};
	\draw (5.5cm,-2em) -- (4cm,-2em) node[midway, below] {c};
	\draw (4cm,-2em) -- (4cm,1em) node[midway, left] {d};
	\end{tikzpicture}
\end{center}

\begin{Th}\label{lucas}
	\begin{enumerate}
		\item There is a one-to-one correspondence between the basis of  $\overline{FK}_{C_n}(n)$ and the set of matchings in an $n$-cycle. 
		\item Dimensions of $\overline{FK}_{C_n}(n)$ equals the number of matchings in an $n$-cycle $C_n$.
	\end{enumerate}
\end{Th}
\begin{proof}
	The basis of the algebra $\overline{FK}_{C_n}(n)$ is the set of all the words on generators\\ $\{x_{12},x_{23},x_{34},\cdots,
	x_{n-1,n}; x_{1,n}\}$ that are not divisible by any leading term of the elements of the Gr\"obner basis in Equation \eqref{gb'}.
	Therefore 
	\begin{itemize}
		\item The degrees $0$ and $1$ elements of  are just $1$ and the generators $x_{1n}$ and $x_{m,m+1},~m=1, \cdots, n-1$.
	\end{itemize}
However for the higher degree terms of $\overline{FK}_{C_n}(n)$, the structure of the degree $2$ terms and the higher degree terms in Gr\"obner basis in Equation \eqref{gb'} is such that: 
	\begin{enumerate}		
		\item Any monomial in our generators vanish by degree $2$ terms in Gr\"obner basis if it has two adjacent $1$st indexes differing by $\pm{1}$ or $0$. 
		\item Any monomial in our generators vanish if two adjacent $1$st indexes differ by $\geq 2$ but are increasing in $1$st indexes. 
		\item If any two adjacent $2$nd indexes in a monomial are equal, then either that common $2$nd index is $n$, that the monomial vanish by $x_{n-1,n}x_{1n}$ or by $x_{1n}x_{n-1,n}$, otherwise if the common index is less than $n$, then so are the $1$st indexes equal as well (as the generators are $x_{m,m+1}$) and so vanishes by $x_{ij}x_{ij}$. 
	\end{enumerate}
	By the above items, the basis elements  of $\overline{FK}_{C_n}(n)$ consist of monomials with decreasing $1$st indexes of variables (with an exception of $x_{1n}$ when appears as the leading variable) such that any two successive $1$st index differ by at least $2$.
	
	Following the above points we come up with $B(\overline{FK}_{C_n}(n))$, the basis of $\overline{FK}_{C_n}(n)$.
	\begin{equation}\label{b}
	\begin{split}
	B&(\overline{FK}_{C_n}(n))=\{1,~x_{1n},~x_{m,m+1},~\text{where}~ 1\leq m \leq n-1,\\&
	x_{1n}x_{m_1m_1+1},~~\text{where}~ 2\leq m_1 \leq n-2,\\&x_{m_1m_1+1}x_{m_2m_2+1}, ~\text{where}~m_1-m_2\geq 2,~m_2\geq1,~m_1\leq n-1,
	\\&x_{m_1m_1+1}x_{m_2m_2+1}\cdots x_{m_km_k+1},~\text{where}~
	m_i-m_{i+1}\geq 2,~m_k\geq 1,~m_1\leq n-1,
	\\&
	x_{1n}x_{m_1m_1+1}x_{m_2m_2+1} \cdots x_{m_{k-1}m_{k-1}+1}, ~\text{where}~
	m_i-m_{i+1}\geq 2,~ m_{k-1}\geq 2,\\&m_1\leq n-2
	\}.
	\end{split}
	\end{equation}
	The following argument shows why $B[\overline{FK}_{C_n}(n)]$ consists of all the matchings in $C_n$.\\
	In the elements of the basis of $\overline{FK}_{C_n}(n)$ in Equation \eqref{b}, if we view variables $x_{m,m+1}$ for $m=1, 2, \cdots, n-1$ and $x_{1n}$, as the edges $\overline{12}, \overline{23},\cdots \overline{n-1,n}$ and $\overline{1n}$ of an $n$-cycle $C_n$, then we can view an element of the basis as a subset of the set of the edges of $C_n$.
	
	An element in the basis of $\overline{FK}_{C_n}(n)$ in \eqref{b}, is a matching in an $n$-cycle $C_n$. The reason is that 
	From $B[\overline{FK}_{C_n}(n)]$ in \eqref{b}, no monomial in the basis includes two variables with a common index, so we can view an element of the basis as a subset of the set of the edges of the graph $C_n$ such that no two edges have a common vertex i.e., a matching in $C_n$. 
	
	Vice versa, we can view a matching in $C_n$ as an element in the basis. 
	
	Also from the domain of the indexes in Equation \eqref{b}, we see that all the monomials in generators of $\overline{FK}_{C_n}(n)$ that have no repeated indexes
	 have been covered in the basis. 
	 
	 The above considerations imply that there is a one-to-one correspondence between the set of all matchings in an $n$-cycle $C_n$ and the basis of $\overline{FK}_{C_n}(n)$. 
	 
	 Hence the dimension of the algebra $\overline{FK}_{C_n}(n)$ is equal to the number of matchings in an $n$-cycle $C_n$. This completes the proof. 
\end{proof}
\begin{Cor}\label{lucasnumber}
	Dimension of $\overline{FK}_{C_n}(n)$ equals Lucas number $L_n$.
\end{Cor} 
\begin{proof}
	From Theorem \ref{lucas} the dimensions of $\overline{FK}_{C_n}(n)$ equals the number of matchings in an $n$-cycle $C_n$, which equals the Lucas number $L_n$ \cite{eric}. This completes the proof.
\end{proof}
\subsection{The highest degree elements in the basis}
From the basis of $\overline{FK}_{C_n}(n)$ in Equation \eqref{b}, the maximum degree of a monomial in the basis (the maximum size of a matching) is attained when any two adjacent 1st indexes differ only by $2$ (equivalently if any two sides in the corresponding matching is separated by exactly one side in $C_n$). Therefore one can calculate easily that 
\begin{equation}\label{maxdeg}\begin{split}\text{For}~n\geq 4, ~ \text{max (degree /size)}~=
\lfloor\frac{n}{2}\rfloor,~ \text{and}\\ \text{multiplicity of}~\text{max (degree /size)}~=\begin{cases}
2, &\text{for even}~n\\
n, &\text{for odd}~n
\end{cases}
\end{split}\end{equation}
\begin{Ex}
	The highest degree for elements in $B[\overline{FK}(5)]$ is $\lfloor\frac{n}{2}\rfloor=\lfloor\frac{5}{2}\rfloor=2$, and its multiplicity $5$. 
\end{Ex} 
\subsection{`Matchings' and Fibonacci number}
\begin{Lemma}\label{fibo1}
	Let $M(n)$ denote the number of matchings in the set of line segments\\ $V_n=\{\overline{12},\overline{23},\cdots \overline{n-1,n}\}$, $n\geq 2$. Then
	\begin{equation}\label{chi5'}
	M(n)=F_n,~ n\geq 2, 
	\end{equation} where $F_n$ is $n$-th Fibonacci number with  initial values $F_0=1,~F_1=1$.
\end{Lemma}
\begin{proof}
	We need to show 
	\begin{enumerate}
		\item $M(2)=F_2$, and
		\item $M(n+1)=M(n)+M(n-1)$, i.e., the recurrence for Fibonacci numbers \cite{carlitz}, \cite{cigler}.
	\end{enumerate}
	The $1$st item is obvious as $M(2)$ is the number of matching in $\{\overline{12}\}$ which equals $2$, including the empty matching, which equals $F(2)$ (as our starting values for Fibonacci sequence are $F_0=1, F_1=1$).
	
	To show the $2$nd item consider that
	\begin{displaymath}
 V_n=\{\overline{12},\overline{23},\cdots \overline{n-1,n}\}= V_{n-1}\cup\overline{n-1,n},
 \end{displaymath} where $V_{n-1}=\{\overline{12},\overline{23},\cdots\overline{n-2, n-1}\}$.  Then addition of $\overline{n,n+1}$ to $V_n$ makes $V_{n+1}$ and increases $M(n)$ by adding new matchings to it. However in order to avoid formation of common index/vertex, these new matchings have to be made by adding $\overline{n,n+1}$ to each matching only in $V_{n-1}$, as $\overline{n,n+1}$ does not have common index/vertex with elements in $V_{n-1}$ but does have common index/vertex $n$ with $\overline{n-1,n}$.
	Therefore the number of new matchings to be added to $M(n)$ is exactly the number of matchings in $V_{n-1}$, i.e., $M(n-1)$. Hence we have
	\begin{displaymath}
	 M(n+1)=M(n)+M(n-1).
	 \end{displaymath}
	This completes the proof. 
\end{proof}
\begin{Lemma}\label{fibo12} \cite{carlitz}, \cite{cigler}.
	Let $M(n,k)$ denote the number of matchings of size $k$ in the set of line segments  $V_n=\{\overline{12},\overline{23},\cdots \overline{n-1,n}\}$. Then 
	\begin{enumerate}
		\item  $M(n+1,k)=M(n,k)+M(n-1,k-1), ~\text{for}~ n\geq 2,~k\geq 0$,
		\item $M(n, k)={n-k\choose k}$.
	\end{enumerate}	
\end{Lemma}
\begin{proof}\begin{enumerate}
		\item 
		Consider $V_n$ as
		\begin{displaymath} V_n=V_{n-1}\cup \{\overline{n-1,n}\},
		\end{displaymath} where		
	 $V_{n-1}=\{\overline{12},\overline{23},\cdots \overline{n-2,n-1}\}$.
	  Addition of the line segment $\overline{n,n+1}$ to $V_n$ makes 
	  \begin{displaymath}
	  V_{n+1}=\{\overline{n,n+1}\}\cup V_{n-1}\cup \{\overline{n-1,n}\}.
	  \end{displaymath} This addition increases $M(n, k)$ by adding $\overline{n,n+1}$ to each element of size $k-1$ in only $V_{n-1}$ without developing common index/vertex, as the elements of $V_{n-1}$ do not have indexes $n$ or $n+1$. So the number of these new matchings to be added to $M(n, k)$ is exactly the number of elements in $V_{n-1}$ of size $k-1$, i.e., $M(n-1, k-1)$. Hence we have
	  \begin{displaymath}
	  M(n+1,k)=M(n,k)+M(n-1,k-1).
		\end{displaymath} 
		\item 
		To prove $M(n, k)={n-k\choose k}$ we use double induction. We need to show that
		\begin{enumerate}
			\item base of induction: For $n=2$ and $k=1$ the statement $M(2,1)={2-1\choose 1}$ is true, as $M(2,1)=1$ and ${2-1\choose 1}={1\choose 1}=1$.
			\item Inductive step: We need to show 		
			 \begin{displaymath}\begin{split}				\text{If}~	&M(n,k)={n-k\choose k} \forall n\geq 2,~k>1,~\text{then}~ M(n+1,k)={n+1-k\choose k}.\\
			M(n+1,k)&=M(n, k)+M(n-1, k-1)\\&={n-k\choose k}+{n-1-(k-1)\choose k-1}={n-k\choose k}+{n-k\choose k-1}\\
			&=\frac{(n-k)!}{k!(n-2k)!}+\frac{(n-k)!}{(k-1)!(n-2k+1)!}=\cdots=\frac{(n-k+1)!}{k!(n-2k+1)!}\\&={n+1-k\choose k}.\\
			\text{and}\\
			\text{if}~				&M(n,k)={n-k\choose k} \forall n\geq 2,~k>1,~\text{then}~ M(n,k+1)={n-k-1\choose k+1}.\\
				M(n,k+1)&=M(n-1, k+1)+M(n-2, k)\\&={n-1-(k+1)\choose k+1}+{n-2-k\choose k}={n-k-2\choose k+1}+{n-k-2\choose k}\\
				&=\frac{(n-k-2)!}{(k+1)!(n-2k-3)!}+\frac{(n-k-2)!}{k!(n-2k-2)!}=\cdots=\frac{(n-k-1)!}{(k+1)!(n-2k-21)!}\\&={n-k-1\choose k+1}
			\end{split}
		\end{displaymath}
		\end{enumerate}
	\end{enumerate}
	This completes the proof.	 
\end{proof}			
	
\begin{Prop}\label{fullHilbert}\cite{carlitz}, \cite{cigler}. Let $M_{C_n}(n,k)$ denote the number of matchings of size $k$ in an $n$-cycle graph $C_n$ and $M(n, k)$ the number of matchings of size $k$ in the set of line segments  $V_n=\{\overline{12},\overline{23},\cdots \overline{n-1,n}\}$. Then we have
	\begin{equation}\label{123'}
	\begin{split}
	M_{C_n}(n, k)=&M(n, k)+M(n-2, k-1),~n\geq 4, ~k=2, 3, \cdots, \lfloor\frac{n}{2}\rfloor,~and\\
	&M_{C_n}(n, k)=\frac{n}{n-k}{n-k\choose k},
	\end{split}
	\end{equation}	
	 where $M(n, 0)=1$, \text{and where}~ $M(n, k)={n-k\choose k}$ is the number of matchings of size $k$ in $C_n\setminus \overline{\{1n\}}$. 
\end{Prop}
\begin{proof}
	$M(n, k)$ could also be viewed as the number of matchings of size $k$ in $C_n\setminus{\{\overline {1n}\}}$. We rewrite $C_n\setminus{\{\overline {1n}\}}$ as
	\begin{displaymath}
 C_n\setminus{\{\overline {1n}\}}=\{\overline{12} \}\cup V_{n-2}\cup \{\overline{n-1,n}\},
 \end{displaymath} where $V_{n-2}=\{\overline{23},\overline{34}, \cdots, \overline{n-2,n-1}\}$. Now adding $\overline{1n}$ to $C_n\setminus{\{\overline {1n}\}}$ completes the $n$-cycle, as well increases $M(n, k)$ to $M_{C_n}(n, k)$ by making new matchings of size $k$. However to avoid making a common vertex we need to add $\overline{1n}$ to the elements of size $k-1$ only in $V_{n-2}$, as they do not involve in indexes $1$ or $n$ but $\overline{12}$ and $\overline{n-1,n}$ do involve in common indexes $1$ or $n$.
	So the number of these new matchings is equal to the number of matchings of size $k-1$ in $V_{n-2}$ i.e., $M(n-2, k-1)$. Hence $M_{C_n}(n, k)=M(n, k)+M(n-2, k-1)$. This completes the proof of the first part of the proposition.\\
	To prove the second part, substituting $M(n, k)={n-k\choose k}$ from Lemma \ref{fibo12} into the first part, we have
	\begin{eqnarray}\begin{split}
	M_{C_n}(n, k)&=M(n, k)+M(n-2, k-1)={n-k\choose k}+{n-1-k\choose k-1}\\&={n-k\choose k}+\frac{(n-1-k)!}{(k-1)!(n-2k)!}={n-k\choose k}+\frac{\frac{(n-k)!}{(n-k)}}{\frac{k!}{k}(n-2k)!}\\&={n-k\choose k}+\frac{k}{n-k}\frac{(n-k)!}{k!(n-2k)!}={n-k\choose k}+\frac{k}{n-k}{n-k\choose k}\\&=\frac{n}{n-k}{n-k\choose k}.
	\end{split}
	\end{eqnarray} This completes the proof. 
\end{proof}
\subsection{Hilbert series and $q$-Fibonacci/$q$-Lucas polynomials}
From Proposition \ref{fullHilbert} and Theorem \ref{lucas} and Corollary \ref{lucasnumber}, we have
\begin{displaymath}
L_n=M_{C_n}(n)=\sum_{k=0}^{\lfloor\frac{n}{2}\rfloor}M_{C_n}(n,k)=\sum_{k=0}^{\lfloor\frac{n}{2}\rfloor}\frac{n}{n-k}{n-k\choose k}.
\end{displaymath} The upper limit in the sum is due to the fact that maximum size of a matching in $C_n$ or the maximum degree of an element in the basis of $\overline{FK}_{C_n}(n)$ is $\lfloor\frac{n}{2}\rfloor$. 

Hence taking sum over all the matchings of different sizes $k$ in $C_n$ (elements of different degrees $k$ in the basis of $\overline{FK}_{C_n}(n)$), we have $q$-Lucas polynomial \cite{carlitz}, \cite{cigler}:
\begin{equation}\label{q-lucas}
L_n(q)=\sum_{k=0}^{\lfloor\frac{n}{2}\rfloor}\frac{n}{n-k}{n-k\choose k}q^k,~L_0=2.
\end{equation}
From Lemmas \ref{fibo1} and \ref{fibo12} we have the total number of matchings in a set of line segments  $V_n=\{\overline{12},\overline{23},\cdots \overline{n-1,n}\}$ to be $M_n=F_n$ and the number of matchings of size $k$ in $V_n$ to be $M(n,k)={n-k\choose k}$. This results in
\begin{displaymath}
F_n=M(n)=\sum_{k=0}^{\lfloor\frac{n}{2}\rfloor}M(n,k)=\sum_{k=0}^{\lfloor\frac{n}{2}\rfloor}{n-k\choose k}.
\end{displaymath}
Hence taking sum over all matchings of different sizes we have $q$-Fibonacci polynomial \cite{carlitz}, \cite{cigler}: 
\begin{equation}\label{q-fibo} F_n(q)=\sum_{k=0}^{\lfloor\frac{n}{2}\rfloor}{n-k\choose k}q^k. 
\end{equation}
\begin{Prop}
	Hilbert series of $\overline{FK}_{C_n}(n)$ is $q$-Lucas number
	\begin{equation}
	H_n=\sum_{k=0}^{\lfloor\frac{n}{2}\rfloor}\frac{n}{n-k}{n-k\choose k}q^k=L_n(q).
	\end{equation}	
\end{Prop}
\begin{proof}
	We recall here \eqref{q-lucas},
	\begin{displaymath}
	L_n(q)=\sum_{k=0}^{\lfloor\frac{n}{2}\rfloor}\frac{n}{n-k}{n-k\choose k}q^k,~L_0=2.
	\end{displaymath}
	In the above expression, degree $k$, as the size of matching, is summed over to give us the Lucas number which is by Corollary
	\ref{lucasnumber} equal to the dimension of $\overline{FK}_{C_n}(n)$. Hence the above expression indeed demonstrates Hilbert series  
	\begin{displaymath}
	H_n=\sum_{k=0}^{\lfloor\frac{n}{2}\rfloor}\frac{n}{n-k}{n-k\choose k}q^k.
	\end{displaymath}
	This completes the proof.
\end{proof}
\begin{Ex}
	$H_5=\sum_{k=0}^2\frac{5}{5-k}{5-k\choose k}q^k={5\choose 0}+\frac{5}{4}{4\choose 1}q+\frac{5}{3}{3\choose 2}q^2=1+5q+5q^2$. Similarly $H_6=1+6q+9q^2+2q^3$.
\end{Ex}
\subsection{The action of $D_n$ and decomposition}
In this section we derive the character of $\overline{FK}_{C_n}(n)$ over $D_n$ for even and odd $n$. 
\begin{Rem}\label{char-cal}
As mentioned before, $D_n$ defines an action on algebra $\overline{FK}_{C_n}(n)$. Under the action of $D_n$, the component of the character of $\overline{FK}_{C_n}(n)$ associated with the conjugacy class indexed by $\sigma\in D_n$, is the sum of the coefficients of $z$ in $\sigma(z)$ when $z$ runs in the basis $B[\overline{FK}_{C_n}(n)]$. 
\end{Rem}
\begin{Lemma}\label{char11}
	Let $r\in D_n$ realized in permutation group $S_n$ as $r=(1,2,\cdots,n)$. Then for integers $2\leq p\leq n$ the trace of representation of $r^p\in D_n$ on $\overline{FK}_{C_n}(n)$ is
	\begin{equation}
	\chi(r^p)=1+p\delta_{n, pm},
	\end{equation}for some $m=\frac{n}{p}$ if it is an integer. In $q$-counting according to the total degree decomposition form we have 
	\begin{equation}\label{qcchar}
	[\chi(r^p)](q)=1+pq^\frac{n}{p}\delta_{\frac{n}{p},\text{integer}}.
	\end{equation} 
\end{Lemma}
\begin{proof}
	One can view a monomial $M\in B[\overline{FK}_{C_n}(n)]$ as a matching in an $n$-cycle (i.e., a subset of the set of edges of the $n$-cycle wherein no two edges have a common vertex). Also we can consider $r$ as a clockwise rotation of angle $\frac{2\pi}{n}$ in $C_n$. Then $r^pM$ is the result of clockwise rotation of the matching $M$ in $n$-cycle by angle $\frac{2\pi p}{n}$. So $r^pM=M$ means that after rotation of a matching $M$ over $p$ edges of the $n$-cycle, the matching $M$ coincides with itself. We discuss the following cases.
	\begin{enumerate}
		\item The possibility that $r^pM=-M$ is absurd, i.e., the operator $r^p$ does not have an eigenvector with eigenvalue $-1$. The reason is that in order for $r^pM$ to develop a minus sign we need $M$ contains an $x_{1n}$ to get rotated into
		\begin{displaymath}
	 r^px_{1n}=x_{1+p,n+p}=x_{1+p,p}=-x_{p,p+1},
	 \end{displaymath} and thereby create a minus sign. Then since $r^p$ rotates $M$ as a matching onto itself, there must be another variable in $M$ like  $x_{n-p, n+1-p}$ to be rotated onto $x_{1n}$ to preserve $M$ as a matching in $C_n$. However then
	 \begin{displaymath}
r^px_{n-p, n+1-p}=x_{n, n+1}=x_{n,1}=-x_{1n};
\end{displaymath} a second minus sign to cancel the first one. So $r^pM=-M$ can not hold.
		\item
		$rM=M$ does not hold for non-empty $M$. The reason is that since $r$ is a clockwise rotation of $\frac{2\pi}{n}$ in $C_n$, it takes an edge in the matching $M$ in $C_n$, into the edge next to it (the neighbor edge) in $C_n$. 
		However this next edge is not in $M$ otherwise there would be a common vertex in $M$ that is forbidden by the definition of a matching. Thus $rM=M$ can not hold for a non-empty matching $M$.
		\item 
		For $2\leq p \leq n$ and a non-empty matching $M$, for $r^pM=M$ to hold, we need that $p$ divides $n$, $p|n$ and the matching $M$ be symmetric in $C_n$ in the sense that any two successive edges in $M$ be separated by the same number  $p-1$ of edges in $C_n$. Therefore the number of non-empty matchings $M$ such that $r^pM=M$, is equal to $p\delta_{n, pm}$ for some $m=\frac{n}{p}$, if it is an integer.
		\item For any $p$ there is always an empty matching (associated for identity element $1$ in the basis) such that $r^pM=M (as~ r^p1=1)$.
	\end{enumerate}
	Considering items $1-4$, we come up with the conclusion that for a given integer $p$ such that $2\leq p\leq n$, the number of matchings (monomials) $M$ such that $r^pM=M$, i.e., the trace of representation of $r^p$ is $1+p\delta_{n, pm}$ for some $m=\frac{n}{p}$ if it is an integer. Therefore for  we have
	\begin{displaymath}
\chi(r^p)=1+p\delta_{n, pm},~2\leq p \leq n,~\text{for some}~ m=\frac{n}{p}~\text{if it is an integer}.
	\end{displaymath} The number $1$ in $\chi(r^p)=1+p\delta_{n, pm}$ refers to the empty matching (degree zero element of the basis) and $p\delta_{n, pm}$ to the nonempty ones of size (degree)  $m=\frac{n}{p}$ if it is an integer.
	
	 Therefore we can write the $q$-counting form of $\chi(r^p)=1+p\delta_{n, pm}$ as
	 \begin{displaymath}
	 [\chi(r^p)](q)=1+pq^\frac{n}{p}\delta_{\frac{n}{p},\text{integer}},
	 \end{displaymath}
	 where $\frac{n}{p}$ is the size of the set of non-empty matchings, that are the eigenvectors of $r^p$. This completes the proof.
\end{proof}	
\begin{Lemma}\label{evens}
	For even $n\geq 4$, let $s\in D_n$, realized in permutation group as\\ $s=(1n)(2, n-1)\cdots (\frac{n}{2}, \frac{n}{2}+1)$ in an $n$-cycle graph $C_n$. Let $\chi(s)$ be the trace of the representation of $s$ over $\overline{FK}_{C_n}(n)$. Then
	\begin{equation}\label{char1}
	[\chi(s)](q)=\sum_{k=0}^{\lfloor\frac{n}{4}\rfloor}{\frac{n}{2}-k\choose k}q^{2k}-2\sum_{k=0}^{\lfloor\frac{\frac{n}{2}-1}{2}\rfloor}{\frac{n}{2}-1-k\choose k}q^{2k+1}+\sum_{k=0}^{\lfloor\frac{\frac{n}{2}-2}{2}\rfloor}{\frac{n}{2}-2-k\choose k}q^{2k+2}.
	\end{equation} In terms of $q$-Fibonacci polynomials, according to the total degree decomposition we have
	\begin{equation}[\chi(s)](q)=F_{\frac{n}{2}}(q^2)-2qF_{\frac{n}{2}-1}(q^2)
	+q^2F_{\frac{n}{2}-2}(q^2).
	\end{equation}
\end{Lemma}
\begin{proof} By Theorem \ref{lucas} to any element of the basis of $\overline{FK}_{C_n}(n)$ there corresponds a matching in $C_n$.
	Let  $T_1=\{\overline{i_1,i_1+1}, \overline{i_2,i_2+1}\cdots \overline{i_m,i_m+1}\}$ be a matching of the set of line segments $S_1=\{\overline{12}, \overline{23}, \cdots,\overline{\frac{n}{2}-1, \frac{n}{2}}\}$ in $C_n$ defined in Definition \ref{matching''}, with its corresponding monomial 	\begin{displaymath}
	u=x_{i_1,i_1+1}x_{i_2,i_2+1}\cdots  x_{i_m,i_m+1}\in B[\overline{FK}_{C_n}(n)].
	\end{displaymath} Let its reflection in $s$, itself a matching (as $n=$even), be
	\begin{displaymath}T'_1=\{\overline{j_1,j_1+1}, \overline{j_2,j_2+1}\cdots \overline{j_m,j_m+1}\},\end{displaymath} and its corresponding monomial be $u'=x_{j_1j_1+1}x_{j_2j_2+1}\cdots $ $\cdots x_{j_mj_m+1}$.
	
	 While the matching $T_1$ is not symmetric about $s$, however matching $T_1\cup T'_1$ is symmetric and its corresponding monomial $uu' \in B[\overline{FK}_{C_n}(n)]$ forms an eigenvector of $s$ with eigenvalue $+1$ as $s(uu')=uu'$. The reason $uu'$ is an eigenvector of $s$ is that 
	 \begin{displaymath}
	 s(uu')=s(usu)=(su)s^2u=(su)u=u'u\to uu'. 
	 \end{displaymath} The last step in the above is by reduction w.r.t. degree $2$ terms in the Gr\"obner basis \eqref{gb'}, because $u$ and $u'$ commute, as there is no common index among variables in $uu'$.	 
	  The number of such eigenvectors in $\overline{FK}_{C_n}(n)$ equals the number of matchings $T_1\cup T_1'$ in $C_n$ which by construction equals the number of matchings $T_1$ in the set $\{\overline{12}, \overline{23}, \cdots,\overline{\frac{n}{2}-1,\frac{n}{2}}\}$. However the number of matchings in $\{\overline{12}, \overline{23}, \cdots,\overline{\frac{n}{2}-1,\frac{n}{2}}\}$ by Equations \eqref{chi5'} and \eqref{q-fibo} is equal to 
	\begin{equation}\label{chi1'}
	[M(\frac{n}{2})](q)=F_\frac{n}{2}(q)=\sum_{k=0}^{\lfloor\frac{n}{4}\rfloor}{\frac{n}{2}-k\choose k}q^{2k},
	\end{equation} where the factor $2$ in the power of $q^{2k}$ is to take care of the double number of sides in $T_1\cup T_1'$ compared to $T_1$ and it reflects the even number of matchings in $T_1\cup T_1'$. Hence our first contribution to $\chi(s)$ denoted $\chi_1(s)$ is 
	\begin{equation}\label{chi1}
	[\chi_1^{}(s)](q)=\sum_{k=0}^{\lfloor\frac{n}{4}\rfloor}{\frac{n}{2}-k\choose k}q^{2k}.
	\end{equation} 
	To find the number of eigenvectors with eigenvalue $-1$  instead of considering the matching $T_1\subset\{\overline{12}, \overline{23}, \cdots,\overline{\frac{n}{2}-1,\frac{n}{2}}\}$ and its reflection $T'_1$ in $s$, we need to consider the following cases
	\begin{enumerate}
		\item The matching $\{\overline{1n}\}\cup T_2$, with matching $T_2\subset\{\overline{23},\cdots, \overline{\frac{n}{2}-1,\frac{n}{2}}\}$ (to prevent developing common vertex with $\{\overline{1n}\}$), and its reflection $T_2'$ in $s$, with the corresponding monomials $v, v'\in B[\overline{FK}_{C_n}(n)]$, to $T_2$ and $T'_2$ respectively.
		\item The matching 		  
		$\{\overline{\frac{n}{2}, \frac{n}{2}+1}\}\cup T_3$ with matching $T_3 \subset\{\overline{12},\cdots, \overline{\frac{n}{2}-2,\frac{n}{2}-1}\}$ (to prevent developing common vertex with $\{\overline{\frac{n}{2}, \frac{n}{2}+1}\}$)
		and its reflection $T_3'$ in $s$.
	\end{enumerate}In case $1$ while the  matching $\{\overline{1n}\}\cup T_2$ is not symmetric about $s$, however $\{\overline{1n}\}\cup T_2\cup T_2'$ is a matching in $C_n$ symmetric about $s$ and its corresponding monomial $x_{1n}vv'$ forms an eigenvector of $s$ with eigenvalue $-1$ as $s(x_{1n}vv')=-x_{1n}s(vv')=-x_{1n}vv'$ as before.

 By construction the number of matchings $\{\overline{1n}\}\cup T_2\cup T_2'$ in $C_n$ equals the number of matchings  $T_2$ in $\{\overline{23},\cdots, \overline{\frac{n}{2}-1,\frac{n}{2}}\}$ which is by Equations \eqref{chi5'} and \eqref{q-fibo} equal to
	\begin{equation}\label{chi2'}
	[M(\frac{n}{2}-1)](q)=F_{\frac{n}{2}-1}(q)=\sum_{k=0}^{\lfloor\frac{\frac{n}{2}-1}{2}\rfloor}{\frac{n}{2}-1-k\choose k}q^{2k+1}.
	\end{equation} Here the odd power of $q^{2k+1}$ reflects the odd size of the matching $\{\overline{1n}\}\cup T_2\cup T_2'$ in $C_n$ as well the negative eigenvalue of $s$.		  
	Considering a factor $2$ to cover the $2$nd case (as it exactly results to the same as case $1$), as well a negative sign for the negative eigenvalue, we come up with our second contribution to $\chi(s)$ 
	\begin{equation}\label{chi2}
	[\chi_2^{}(s)](q)=-2\sum_{k=0}^{\lfloor\frac{\frac{n}{2}-1}{2}\rfloor}{\frac{n}{2}-1-k\choose k}q^{2k+1}.
	\end{equation} 
	To cover the last contribution to $\chi(s)$, we consider the matching 
	\begin{displaymath}
\{\overline{1n}\}\cup\{\frac{n}{2}, \frac{n}{2}+1\}\cup T_4,
\end{displaymath} and $T'_4$, the reflection of $T_4$ in $s$, with matching $T_4\subset\{\overline{23},\overline{34},\cdots, \overline{\frac{n}{2}-2,\frac{n}{2}-1}\}$ (to prevent developing common vertex with $\overline{1n}$ or $\overline{\frac{n}{2}, \frac{n}{2}+1}~)$ with $w$ and $w'$ the corresponding monomials to $T_4$ and $T_4'$ respectively.
	By similar arguments we have eigenvectors with eigenvalues $+1$, with the number of such eigenvectors equal to the number of matchings
	$T_4\subset\{\overline{23},\overline{34}, \cdots,\overline{\frac{n}{2}-2,\frac{n}{2}-1}\}$ as before,
	which is by Equations \eqref{chi5'} and \eqref{q-fibo} equal to
	\begin{equation}\label{chi3'}
	[M(\frac{n}{2}-2)](q)=F_{\frac{n}{2}-2}(q)=\sum_{k=0}^{\lfloor\frac{\frac{n}{2}-2}{2}\rfloor}{\frac{n}{2}-2-k\choose k}q^{2k+2}.
	\end{equation} Hence our last contribution to $\chi(s)$ is 
	\begin{equation}\label{chi3}
	[\chi_3^{}(s)](q)=\sum_{k=0}^{\lfloor\frac{\frac{n}{2}-2}{2}\rfloor}{\frac{n}{2}-2-k\choose k}q^{2k+2}.
	\end{equation} Now adding the contributions in Equations \eqref{chi1},
	\eqref{chi2} and \eqref{chi3} we come up with 
	\begin{displaymath}
	[\chi(s)](q)=\sum_{k=0}^{\lfloor\frac{n}{4}\rfloor}{\frac{n}{2}-k\choose k}q^{2k}-2\sum_{k=0}^{\lfloor\frac{\frac{n}{2}-1}{2}\rfloor}{\frac{n}{2}-1-k\choose k}q^{2k+1}+\sum_{k=0}^{\lfloor\frac{\frac{n}{2}-2}{2}\rfloor}{\frac{n}{2}-2-k\choose k}q^{2k+2},
	\end{displaymath} which is written in the following $q$-Fibonacci form using Equation $\eqref{q-fibo}$: 	\begin{displaymath}
	[\chi(s)](q)=F_{\frac{n}{2}}(q^2)-2qF_{\frac{n}{2}-1}(q^2)
	+q^2F_{\frac{n}{2}-2}(q^2).	\end{displaymath} This completes the proof.
\end{proof}	  
\begin{Ex}
	For $n=6$ we have $[\chi(s)](q)=F_3(q^2)-2qF_2(q^2)
	+q^2F_1(q^2)$.\\ Let $q=1$, then~ $\chi(s)\to F_3-2F_2+F_1=3-2(2)+1=0$.
\end{Ex}
\begin{Lemma}\label{char22}
	For even $n$, Let $s, r\in D_n$, realized in permutation group $S_n$ as\\ $s=(1,n)(2, n-1)\cdots (\frac{n}{2}, \frac{n}{2}+1)$ and $r=(1,2,\cdots,n)$ in an $n$-cycle graph $C_n$ . Then
	\begin{equation}\label{char3}
	[\chi(sr)](q)=\sum_{k=0}^{\lfloor\frac{\frac{n}{2}-1}{2}\rfloor}{\frac{n}{2}-1-k\choose k}q^{2k}, ~n\geq 4,
	\end{equation}
	and in $q$-Fibonacci form according to total degree decomposition we have:
	\begin{equation}[\chi(sr)](q)=F_{\frac{n}{2}-1}(q^2),~n\geq 4,
	\end{equation}
	where $\chi(sr)$ is the trace of the representation of $sr$ over $\overline{FK}_{C_n}(n)$.	
\end{Lemma}
\begin{proof}
	Here we have
	\begin{displaymath}
 sr=(n)(1,n-1)(2,n-2)\cdots (\frac{n}{2}-1,\frac{n}{2}+1)(\frac{n}{2}).
 \end{displaymath} Therefore the only matchings in the $n$-cycle $C_n$ symmetric about $sr$ are $T\cup T'$ with matching $T\subset\{\overline{12},\overline{23},\cdots, \overline{\frac{n}{2}-2,\frac{n}{2}-1}\}$ and $T'$, the reflection of $T$ in $sr$. The corresponding monomials to $T$ and $T'$ are respectively $w$ and $w'$. 
 
 we have eigenvectors of $sr$ with eigenvalue $+1$, as $sr(ww')=ww'$ as before. The number of such eigenvectors equals the number of matchings $T\cup T'$ in $C_n$ which equals the number of matchings $T\subset\{\overline{12},\overline{23},\cdots, \overline{\frac{n}{2}-2,\frac{n}{2}-1}\}$ which is by Equations \eqref{chi5'} and \eqref{q-fibo} equal to
	\begin{displaymath}[\chi(sr)](q)=M(\frac{n}{2}-1)(q)=F_{\frac{n}{2}-1}(q)=\sum_{k=0}^{\lfloor\frac{\frac{n}{2}-1}{2}\rfloor}{\frac{n}{2}-1-k\choose k}q^{2k}.
	\end{displaymath}
	This completes the proof.	
\end{proof}
\begin{Ex}
	For $n=6$, we have $\chi(sr)=\sum_{k=0}^1{2-k\choose 1}q^{2k}={2\choose 1}+{1\choose 1}q^2=1+q^2$.
\end{Ex}
\begin{Lemma}\label{odds}
	For odd $n$, Let $s\in D_n$ realized in permutation group $S_n$ as\\ $s=(1,n)(2, n-1)\cdots (\frac{n-1}{2}, \frac{n+3}{2})(\frac{n+1}{2})$ in an $n$-cycle graph $C_n$. Then
	\begin{equation}\label{char2}
	[\chi(s)](q)=\sum_{k=0}^{\lfloor\frac{n-1}{4}\rfloor}{\frac{n-1}{2}-k\choose k}q^{2k}-\sum_{k=0}^{\lfloor\frac{n-3}{4}\rfloor}{\frac{n-3}{2}-k\choose k}q^{2k+1},
	\end{equation}
	with $q$-Fibonacci version
	\begin{displaymath}
	[\chi(s)](q)=F_\frac{n-1}{2}(q^2)-qF_\frac{n-3}{2}(q^2),
	\end{displaymath}
	where $\chi(s)$ is the trace of the representation of $s$ over $\overline{FK}_{C_n}(n)$.
\end{Lemma}	
\begin{proof}
	Let
	\begin{displaymath}	
	T_1=\{\overline{i_1,i_1+1}, \overline{i_2,i_2+1}\cdots \overline{i_m,i_m+1}\}
	\end{displaymath} be a matching of the set of line segments
	\begin{displaymath}
 S=\{\overline{12}, \overline{23}, \cdots,\overline{\frac{n-1}{2}-1, \frac{n-1}{2}}\},
 \end{displaymath} with its corresponding monomial 
	\begin{displaymath}
	u=x_{i_1,i_1+1}x_{i_2,i_2+1}\cdots  x_{i_m,i_m+1}.
	\end{displaymath} Let its reflection in $s$, itself a matching, be 	
	\begin{displaymath}T'_1=\{\overline{j_1,j_1+1}, \overline{j_2,j_2+1}\cdots \overline{j_m,j_m+1}\},\end{displaymath} and its corresponding monomial be
	\begin{displaymath}
 u'=x_{j_1j_1+1}x_{j_2j_2+1}\cdots x_{j_mj_m+1}.
 \end{displaymath}
	
	While the matching $T_1$ is not symmetric about $s$, however matching $T_1\cup T'_1$ is so, and its corresponding monomial $uu'$ forms an eigenvector of $s$ with eigenvalue $+1$, as $s(uu')=uu'$ as before. 
	
	The number of such eigenvectors equals the number of the matchings $T_1\cup T_1'$ in $C_n$ which equals the number of matchings $T_1$ in $\{\overline{12}, \overline{23}, \cdots,\overline{\frac{n-1}{2}-1,\frac{n-1}{2}}\}$ which is by \eqref{chi5'} and \eqref{q-fibo} equal to
	\begin{equation}\label{char2'}
	[M(\frac{n-1}{2})](q)=F_\frac{n-1}{2}(q)=\sum_{k=0}^{\lfloor\frac{n-1}{4}\rfloor}{\frac{n-1}{2}-k\choose k}q^{2k},
	\end{equation} where the factor $2$ in the power of $q^{2k}$ as before counts the double number of sides (double size$=$double degree). Hence our first contribution to $\chi(s)$ is 
	\begin{equation}\label{char3_''}
	[\chi_1^{}(s)](q)=\sum_{k=0}^{\lfloor\frac{n-1}{4}\rfloor}{\frac{n-1}{2}-k\choose k}q^{2k}.
	\end{equation}
	We have another set of matchings symmetric about $s$. These matchings are made by adding the  $\overline{1n}$ (its corresponding variable $x_{1n}$), to avoid common vertex with $\overline{1n}$, only to each matching 
	\begin{displaymath}
T_2\subset\Bigg\{\overline{23},\overline{34}, \cdots,\overline{\frac{n-1}{2}-1,\frac{n-1}{2}}\Bigg\}.	
\end{displaymath} Let the reflection of $T_2$ in $s$ be $T_2'$, then $\{\overline{1n}\}\cup T_2\cup T_2'$ is symmetric about $s$ and its correspondent monomial $x_{12}ww'$ is an eigenvector of $s$ with eigenvalue $-1$ as $s(x_{1n}ww')=-x_{12}s(ww')=-x_{12}ww'$ as before.

 The number of such eigenvectors is equal to the number of matchings $\{\overline{1n}\}\cup T_2\cup T_2'$ in $C_n$ which is equal to the number of matchings $T_2\subset\{\overline{23},\overline{34}, \cdots,\overline{\frac{n-1}{2}-1,\frac{n-1}{2}}\}$ which by Equations \eqref{chi5'} and \eqref{q-fibo} is equal to
	\begin{displaymath}
 M(\frac{n-1}{2}-1)=M(\frac{n-3}{2})=F_\frac{n-3}{2}=\sum_{k=0}^{\lfloor\frac{n-3}{4}\rfloor}{\frac{n-3}{2}-k\choose k}q^{2k+1}.
\end{displaymath}	
	Hence our second contribution to $\chi(s)$ is
	\begin{equation}\label{char2''}
	[\chi_2^{}(s)](q)=-\sum_{k=0}^{\lfloor\frac{n-3}{4}\rfloor}{\frac{n-3}{2}-k\choose k}q^{2k+1}.
	\end{equation} Adding Equations \eqref{char3_''} and \eqref{char2''} we come up with
	\begin{displaymath}
	[\chi(s)](q)=\sum_{k=0}^{\lfloor\frac{n-1}{4}\rfloor}{\frac{n-1}{2}-k\choose k}q^{2k}-\sum_{k=0}^{\lfloor\frac{n-3}{4}\rfloor}{\frac{n-3}{2}-k\choose k}q^{2k+1},
	\end{displaymath}
	with its $q$-Fibonacci form according to total degree decomposition 
	\begin{displaymath}
	[\chi(s)](q)=F_\frac{n-1}{2}(q^2)-qF_\frac{n-3}{2}(q^2).
	\end{displaymath}
	This completes the proof.
\end{proof}
\begin{Ex}
	For $n=5$ we have $[\chi(s)](q)=\sum_{k=0}^1{2-k\choose k}q^{2k}-\sum_{k=0}^0{1-k\choose k}q^{2k+1}={2\choose 0}+{1\choose 1}q^2-{1\choose 0}q=1-q+q^2$.\\Let $q=1$, then $\chi(s)=1$.
\end{Ex}
\begin{Rem}
	For odd $n$, there is no symmetry about $sr$ in $C_n$, so we do not discuss the case of $\chi(sr)$. 
\end{Rem}
\begin{table}\begin{center}\caption{Character table of $D_n$ for even $n$. The trivial representation and the three reflections are ~$1$-dimensional, the rest are $2$-dimensional rotations.}
		\label{ct-even}
		\begin{tabular}
			{|c|c|c|c|c|c|c|}
			\hline
			&$\epsilon$&$[r]$&$\vtop{\hbox{\strut$[r^p],~~ p|n$,}\hbox{\strut$2\leq p\leq\frac{n}{2}-1$}}$&$[r^\frac{n}{2}]$&$[s]$&$[sr]$\\
			\hline
			&$\vtop{\hbox{\strut$|c_{1^n}|$}\hbox{\strut$=1$}}$&$\vtop{\hbox{\strut$|c_n|$}\hbox{\strut$=2$}}$&$|c_{(\frac{n}{p})^p}|=2$&$\vtop{\hbox{\strut$|c_{2^\frac{n}{2}}|$}\hbox{\strut$=1$}}$&$\vtop{\hbox{\strut$|c_{2^{\frac{n}{2}'}}|$}\hbox{\strut$=\frac{n}{2}$}}$&$\vtop{\hbox{\strut$|c_{2^{(\frac{n}{2}-1)}11}|$}\hbox{\strut$=\frac{n}{2}$}}$\\
			\hline
			$\chi^{}_{1^n}$&$1$&$1$&$1$&$1$&$1$&$1$\\
			\hline 
			$\chi_{2^{\frac{n}{2}}}$&$1$&$1$&$1$&$1$&$-1$&$-1$\\
			\hline 
			$\chi_{2^{\frac{n}{2}'}}$&$1$&$-1$&$(-1)^p$&$(-1)^{\frac{n}{2}}$&$1$&$-1$\\
			\hline 
			$\chi_{2^{(\frac{n}{2}-1)}11}$&$1$&$-1$&$(-1)^p$&$(-1)^{\frac{n}{2}}$&$-1$&$1$\\
			\hline 
			\vtop{\hbox{\strut $\chi_{(\frac{n}{h})^h},~~h|n,$}\hbox{\strut $1\leq h\leq\frac{n}{2}-1$}} &$2$&$2cos\frac{2h\pi}{n}$&$2cos\frac{2hp\pi}{n}$&$2(-1)^h$&$0$&$0$\\
			\hline		
	\end{tabular}\end{center}
\end{table}
\section{Character of $\overline{FK}_{C_n}(n)$ over $D_n$ and decomposition}
We summarize the results of Lemmas \ref{char11}, \ref{evens}, and \ref{char22} in Equations \eqref{ceven1} and \eqref{ceven1f}. 
\begin{Rem}
	Table \ref{ct-even}, and Equation \eqref{ceven1}, are character table of $D_n$, and character of $\overline{FK}_{C_n}(n)$ for even $n$, with its $q$-Fibonacci form Equation \eqref{ceven1f}. In table \ref{ct-even}, $[x]$ stands the conjugacy class represented by $x$, and the $D_n$ conjugacy classes and the characters of the irreducible representations of $D_n$ are indexed by the cycle type of the associated $D_n$ conjugacy classes. 
\subsection{Character of $\overline{FK}_{C_n}(n)$ over $D_n$ for even $n$}\label{chareven}
\begin{equation}\label{ceven1}\begin{split}
char_{D_n}&[\overline{FK}_{C_n}(n=\text{even})]=\Bigg(\sum_{k=0}^{\frac{n}{2}}\frac{n}{n-k}{n-k\choose k}q^k,~ 1,~ 1+pq^\frac{n}{p}\delta_{\frac{n}{p},\text{integer}},~ 1+\frac{n}{2}q^2,\\&\sum_{k=0}^{\lfloor\frac{n}{4}\rfloor}{\frac{n}{2}-k\choose k}q^{2k}-2\sum_{k=0}^{\lfloor\frac{\frac{n}{2}-1}{2}\rfloor}{\frac{n}{2}-1-k\choose k}q^{2k+1}+\sum_{k=0}^{\lfloor\frac{\frac{n}{2}-2}{2}\rfloor}{\frac{n}{2}-2-k\choose k}q^{2k+2}, \\&\sum_{k=0}^{\lfloor\frac{n-2}{4}\rfloor}{\frac{n-2}{2}-k\choose k}q^{2k}\Bigg).
\end{split}\end{equation}
In terms of $q$-Fibonacci/$q$-Lucas polynomials, with initial values $F_0=1$, $F_1=1$, $L_0=2$, $L_1=1$ we have
\begin{equation}\label{ceven1f}
\begin{split}
char_{D_n}[\overline{FK}_{C_n}(n=\text{even})]=&\Bigg(L_n(q),~1,~1+pq^\frac{n}{p}\delta_{\frac{n}{p},\text{integer}},~1+\frac{n}{2}q^2,\\~~&F_\frac{n}{2}(q^2)-2qF_{\frac{n}{2}-1}(q^2)+q^2F_{\frac{n}{2}-2}(q^2),~F_\frac{n-2}{2}(q^2) \Bigg)
\end{split}
\end{equation}

\end{Rem}
\subsection{Decomposition in irreducible characters}
\begin{Rem} In decomposition of a character $\chi$ in irreducible characters $\chi^{(i)}$ of a group $G$, the coefficients of $\chi^{(i)}$ are derived as in the following formula \cite{sagan}.
	\begin{equation}
	\chi=\sum_im_i\chi^{(i)},~~ m_i=\langle\chi,\chi^{(i)}\rangle=\frac{1}{|G|}\sum_{K}|K|\chi_K^{}\overline{\chi}_K^{(i)},
	\end{equation}where the sum is over conjugacy classes.
\end{Rem}
Applying the above we have
\begin{equation}\label{char5}\begin{split}
char_{D_n}&[\overline{FK}_{C_n}(n=\text{even})]=\\&m_{1^n}\chi_{1^n}+m_{2^{\frac{n}{2}'}}\chi_{2^{\frac{n}{2}'}}+m_{2^{(\frac{n}{2}-1)}11}\chi_{2^{(\frac{n}{2}-1)}11}+m_{2^{\frac{n}{2}}}\chi_{2^{\frac{n}{2}}}+\sum_{1\leq h\leq
	\frac{n}{2}-1,~h|n}m_{(\frac{n}{h})^h}\chi_{(\frac{n}{h})^h},
\end{split}\end{equation}
where the coefficients of the irreducible characters are calculated using character table \ref{ct-even} and either the  character in Equation \eqref{ceven1}, or in terms of $q$-Fibonacci from Equation \eqref{ceven1f}.

By Equation \eqref{ceven1f} we have: 
\begin{equation}\label{coef''}\begin{split}
m_{1^n}&=\frac{1}{2n}\Bigg\{L_n(q)+2+2\sum_{p=2}^{\frac{n}{2}-1}(1+pq^\frac{n}{p}\delta_{\frac{n}{p},\text{integer}})+(1+\frac{n}{2}q^2)\\&+\frac{n}{2}\Big[F_{\frac{n}{2}}(q^2)-2qF_{\frac{n}{2}-1}(q^2)+q^2F_{\frac{n}{2}-2}(q^2)\Big]+\frac{n}{2}F_{\frac{n-2}{2}}(q^2)\Bigg\}.\\
m_{2^{\frac{n}{2}}}&=\frac{1}{2n}\Bigg\{L_n(q)+2+2\sum_{p=2}^{\frac{n}{2}-1}(1+pq^\frac{n}{p}\delta_{\frac{n}{p},\text{integer}})+(1+\frac{n}{2}q^2)\\&-\frac{n}{2}\Big[F_{\frac{n}{2}}(q^2)-2qF_{\frac{n}{2}-1}(q^2)+q^2F_{\frac{n}{2}-2}(q^2)\Big]-\frac{n}{2}F_{\frac{n-2}{2}}(q^2)\Bigg\}.\\
m_{2^{\frac{n}{2}'}}&=
\frac{1}{2n}\Bigg\{L_n(q)-2+2\sum_{p=2}^{\frac{n}{2}-1}(-1)^p(1+pq^\frac{n}{p}\delta_{\frac{n}{p},\text{integer}})+(-1)^{\frac{n}{2}}(1+\frac{n}{2}q^2)\\&+\frac{n}{2}\Bigg[F_\frac{n}{2}(q^2)-2qF_{\frac{n}{2}-1}(q^2)+q^2F_{\frac{n}{2}-2}(q^2)\Bigg]-\frac{n}{2}F_{\frac{n-2}{2}}(q^2)\Bigg\}.\\
m_{2^{(\frac{n}{2}-1)}11}&=
\frac{1}{2n}\Bigg\{L_n(q)-2+2\sum_{p=2}^{\frac{n}{2}-1}(-1)^p(1+pq^\frac{n}{p}\delta_{\frac{n}{p},\text{integer}})+(-1)^{\frac{n}{2}}(1+\frac{n}{2}q^2)\\&-\frac{n}{2}\Bigg[F_{\frac{n}{2}}(q^2)-2qF_{\frac{n}{2}-1}(q^2)+q^2F_{\frac{n}{2}-2}(q^2)\Bigg]+\frac{n}{2}F_{\frac{n-2}{2}}(q^2)\Bigg\}.\\
m_{(\frac{n}{h})^h}&=\frac{1}{2n}\Bigg\{2L_n+4cos\frac{2h\pi}{n}+4\sum_{p=2}^{\frac{n}{2}-1}(1+pq^\frac{n}{p}\delta_{\frac{n}{p},\text{integer}})cos\frac{2hp\pi}{n}\\&+2(-1)^h(1+\frac{n}{2}q^2)\Bigg\},~1\leq h\leq\frac{n}{2}-1, ~\text{such that}~ h|n.
\end{split}\end{equation}
By Equation \ref{ceven1} we get the following version of coefficients 
\begin{equation}\label{coef}\begin{split}
m_{1^n}&=\frac{1}{2n}\Bigg\{\sum_{k=0}^{\frac{n}{2}}\frac{n}{n-k}{n-k\choose k}q^k+2+2\sum_{p=2}^{\frac{n}{2}-1}(1+pq^\frac{n}{p}\delta_{\frac{n}{p},\text{integer}})+(1+\frac{n}{2}q^2)\\&+\frac{n}{2}\Bigg[\sum_{k=0}^{\lfloor\frac{n}{4}\rfloor}{\frac{n}{2}-k\choose k}q^{2k}-2\sum_{k=0}^{\lfloor\frac{\frac{n}{2}-1}{2}\rfloor}{\frac{n}{2}-1-k\choose k}q^{2k+1}+\sum_{k=0}^{\lfloor\frac{\frac{n}{2}-2}{2}\rfloor}{\frac{n}{2}-2-k\choose k}q^{2k+2}\Bigg]\\&+\frac{n}{2}\sum_{k=0}^{\lfloor\frac{n-2}{4}\rfloor}
{\frac{n-2}{2}-k\choose k}q^{2k}\Bigg\}.\\
m_{2^{\frac{n}{2}}}&=
\frac{1}{2n}\Bigg\{\sum_{k=0}^{\frac{n}{2}}\frac{n}{n-k}{n-k\choose k}q^k+2+2\sum_{p=2}^{\frac{n}{2}-1}(1+pq^\frac{n}{p}\delta_{\frac{n}{p},\text{integer}})+(1+\frac{n}{2}q^2)\\&-\frac{n}{2}\Bigg[\sum_{k=0}^{\lfloor\frac{n}{4}\rfloor}{\frac{n}{2}-k\choose k}q^{2k}-2\sum_{k=0}^{\lfloor\frac{\frac{n}{2}-1}{2}\rfloor}{\frac{n}{2}-1-k\choose k}q^{2k+1}+\sum_{k=0}^{\lfloor\frac{\frac{n}{2}-2}{2}\rfloor}{\frac{n}{2}-2-k\choose k}q^{2k+2}\Bigg]\\&-\frac{n}{2}\sum_{k=0}^{\lfloor\frac{n-2}{4}\rfloor}
{\frac{n-2}{2}-k\choose k}q^{2k}\Bigg\}.\\
m_{2^{\frac{n}{2}'}}&=
\frac{1}{2n}\Bigg\{\sum_{k=0}^{\frac{n}{2}}\frac{n}{n-k}{n-k\choose k}q^k-2+2\sum_{p=2}^{\frac{n}{2}-1}(-1)^p(1+pq^\frac{n}{p}\delta_{\frac{n}{p},\text{integer}})+(-1)^{\frac{n}{2}}(1+\frac{n}{2}q^2)\\&+\frac{n}{2}\Bigg[\sum_{k=0}^{\lfloor\frac{n}{4}\rfloor}{\frac{n}{2}-k\choose k}q^{2k}-2\sum_{k=0}^{\lfloor\frac{\frac{n}{2}-1}{2}\rfloor}{\frac{n}{2}-1-k\choose k}q^{2k+1}+\sum_{k=0}^{\lfloor\frac{\frac{n}{2}-2}{2}\rfloor}{\frac{n}{2}-2-k\choose k}q^{2k+2}\Bigg]\\&-\frac{n}{2}\sum_{k=0}^{\lfloor\frac{n-2}{4}\rfloor}
{\frac{n-2}{2}-k\choose k}q^{2k}\Bigg\}.\\
m_{2^{(\frac{n}{2}-1)}11}&=
\frac{1}{2n}\Bigg\{\sum_{k=0}^{\frac{n}{2}}\frac{n}{n-k}{n-k\choose k}q^k-2+2\sum_{p=2}^{\frac{n}{2}-1}(-1)^p(1+pq^\frac{n}{p}\delta_{\frac{n}{p},\text{integer}})+(-1)^{\frac{n}{2}}(1+\frac{n}{2}q^2)\\&-\frac{n}{2}\Bigg[\sum_{k=0}^{\lfloor\frac{n}{4}\rfloor}{\frac{n}{2}-k\choose k}q^{2k}-2\sum_{k=0}^{\lfloor\frac{\frac{n}{2}-1}{2}\rfloor}{\frac{n}{2}-1-k\choose k}q^{2k+1}+\sum_{k=0}^{\lfloor\frac{\frac{n}{2}-2}{2}\rfloor}{\frac{n}{2}-2-k\choose k}q^{2k+2}\Bigg]\\&+\frac{n}{2}\sum_{k=0}^{\lfloor\frac{n-2}{4}\rfloor}
{\frac{n-2}{2}-k\choose k}q^{2k}\Bigg\}.\\
m_{(\frac{n}{h})^h}&=\frac{1}{2n}\Bigg\{2\sum_{k=0}^{\frac{n}{2}}\frac{n}{n-k}{n-k\choose k}q^k+4cos\frac{2h\pi}{n}+4\sum_{p=2}^{\frac{n}{2}-1}(1+pq^\frac{n}{p}\delta_{\frac{n}{p},\text{integer}})cos\frac{2hp\pi}{n}\\&+2(-1)^h(1+\frac{n}{2}q^2)\Bigg\},~1\leq h\leq\frac{n}{2}-1, ~\text{such that} ~h|n.
\end{split}\end{equation}
\begin{Ex}Let\\ $char_{C_6}[\overline{FK}_{C_6}(6, q)]=m_{1^6}\chi^{}_{1^6}+m_{222}\chi^{}_{222}+m_{222'}\chi^{}_{222'}+m_{2211}\chi^{}_{2211}+m_{6}\chi^{}_{6}+m_{33}\chi^{}_{33}$.
	Then from  Equation \eqref{coef} we have
	\begin{displaymath}
	m_{1^6}=1+2q^2, m_{222}=q+q^3, m_{222'}=q^2,~m_{2211}=q+q^3,~m_6=q+q^2,~m_{33}=q+2q^2.
	\end{displaymath} 
	Then substitution of the above coefficients into the character in Equation \eqref{char5} yields:
	\begin{equation}\label{char_'}\begin{split}char_{D_6}[\overline{FK}_{C_6}(6,~ q)]=&(1+2q^2)\chi^{}_{1^6}+(q+q^3)\chi^{}_{222}+(q^2)\chi^{}_{222'}\\&+(q+q^3)\chi^{}_{2211}+(q+q^2)\chi^{}_{6}+(q+2q^2)\chi^{}_{33}.
	\end{split}\end{equation} 	
	Sorting in $q$ we have the decomposition of $char_{D_6}[\overline{FK}_{C_6}(6)]$ in usual degrees.
	\begin{equation}\label{dec_'}\begin{split}
	char_{D_6}[\overline{FK}_{C_6}(6)]=&\chi^{}_{1^6}+(\chi^{}_{222}+\chi^{}_{2211}+\chi^{}_{6}+\chi^{}_{33})q\\&+(2\chi^{}_{1^6}+\chi^{}_{222'}+\chi^{}_{6}+2\chi^{}_{33})q^2+(\chi^{}_{222}+\chi^{}_{2211})q^3,
	\end{split}\end{equation}
	which upon putting $q=1$ reduces to 	
	\begin{displaymath}
	char_{D_6}\overline{FK}
	_{C_6}(6) =3\chi_{1^6}^{}+2\chi_{222}^{}+\chi_{222'}^{}+2\chi_{2211}^{}+2\chi_{6}^{}+3\chi_{33}^{}.
	\end{displaymath}
\end{Ex}
\subsection{Character of $\overline{FK}_{C_n}(n)$ over $D_n$ for odd $n$ and decomposition}\label{charodd}
We summarize the results of Lemmas \ref{char11} and \ref{char22} in Equations \eqref{codd} and \eqref{codd'} as the character of $\overline{FK}_{C_n}(n)$ over $D_n$ for odd $n$.
\begin{equation}\label{codd}\begin{split}
char_{D_n}[\overline{FK}_{C_n}(n=\text{odd})]=&\Bigg(
\sum_{k=0}^\frac{n-1}{2}\frac{n}{n-k}{n-k\choose k}q^k,~ 1,~ 1+pq^\frac{n}{p}\delta_{\frac{n}{p},\text{integer}},\\&\sum_{k=0}^{\lfloor\frac{n-1}{4}\rfloor}{\frac{n-1}{2}-k\choose k}q^{2k}-\sum_{k=0}^{\lfloor\frac{n-3}{4}\rfloor}{\frac{n-3}{2}-k\choose k}q^{2k+1}\Bigg).
\end{split}\end{equation}
In terms of $q$-Fibonacci/$q$-Lucas polynomials, with initial values $F_0=1$, $F_1=1$, $L_0=2$, $L_1=1$ we have
\begin{equation}\label{codd'}\begin{split}
char_{D_n}[\overline{FK}_{C_n}(n=\text{odd})]=&\Bigg(
L_n(q),~1,~1+pq^\frac{n}{p}\delta_{\frac{n}{p},\text{integer}},~F_\frac{n-1}{2}(q^2)-qF_\frac{n-3}{2}(q^2)
\Bigg).
\end{split}\end{equation}


\subsection{Decomposition in irreducible characters}
\begin{equation}\label{char5'}
char_{D_n}[\overline{FK}_{C_n}(n=\text{odd})]=m_{1^n}\chi^{}_{1^n}+m_{2^{\frac{n-1}{2}}}\chi_{2^{\frac{n-1}{2}}}+\sum_{0<h\leq
	\frac{n-1}{2},~h|n}m_{(\frac{n}{h})^h}\chi^{}_{(\frac{n}{h})^h},
\end{equation}
with the following coefficients calculated as before by using irreducible characters of $D_n$ for odd $n$ in Table \ref{irrchar2} and the character of $\overline{FK}_{C_n}(n)$ either from Equation \eqref{codd} or in terms of $q$-Fibonacci in Equation \eqref{codd'}.
\begin{table}\begin{center}\caption{character table of $D_n$ for odd $n$. The trivial and one reflection, the rest are $2$-dimensional rotations.
		}\label{irrchar2}
		\begin{tabular}
			{|c|c|c|c|c|}
			\hline
			&$\epsilon$&$[r]$&$\vtop{\hbox{\strut$[r^p],~ 2\leq p\leq \frac{n-1}{2}$}\hbox{\strut such that $p|n$}}$&$[s]$\\
			\hline
			&$|c^{}_{1^n}|=1$&$|c_n|=2$&$|c_{(\frac{n}{p})^p}|=2$&$|c_{2^{\frac{n-1}{2}}1}|=n$\\
			\hline
			$\chi^{}_{1^n}$&$1$&$1$&$1$&$1$\\
			\hline
			$\chi_{2^{\frac{n-1}{2}}1}$&$1$&$1$&$1$&$-1$\\
			\hline
			$\vtop{\hbox{\strut$\chi^{}_{(\frac{n}{h})^h}, 1\leq h\leq \frac{n-1}{2}$}\hbox{\strut such that $h|n$}}$
			&$2$&$2cos\frac{2h\pi}{n}$&$2cos\frac{2hp\pi}{n}$&$0$\\
			\hline
		\end{tabular}
	\end{center}
\end{table}

From Equation \ref{codd'} we have 
\begin{equation}\begin{split}
m_{1^n}=&\frac{1}{2n}\Bigg[L_n(q)+2+2\sum_{p=2}^\frac{n-1}{2}(1+pq^{\frac{n}{p}}\delta_{\frac{n}{p},integer})+n\Big(F_{\frac{n-1}{2}}(q^2)-qF_{\frac{n-3}{2}}(q^2)\Big)\Bigg].\\
m_{2^\frac{n-1}{2}1}=&\frac{1}{2n}\Bigg[L_n(q)+2+2\sum_{p=2}^\frac{n-1}{2}(1+pq^{\frac{n}{p}}\delta_{\frac{n}{p},integer})-n\Big(F_{\frac{n-1}{2}}(q^2)-qF_{\frac{n-3}{2}}(q^2)\Big)\Bigg].\\
m_{(\frac{n}{h})^h}=&\frac{1}{2n}\Bigg[2L_n(q)+4cos(\frac{2h\pi}{n})+4\sum_{p=2}^\frac{n-1}{2}(1+pq^{\frac{n}{p}}\delta_{\frac{n}{p},integer})cos(\frac{2hp\pi}{n})\Bigg],\\&\text{where}~1\leq h\leq\frac{n-1}{2},~h|n.
\end{split}\end{equation}
From Equation \ref{codd} we have
\begin{equation}\label{coef2}\begin{split}
m_{1^n}=&\frac{1}{2n}\Bigg\{\sum_{k=0}^\frac{n-1}{2}\frac{n}{n-k}{n-k\choose k}q^k+2+2\sum_{p=2}^{\frac{n-1}{2}}(1+pq^\frac{n}{p}\delta_{\frac{n}{p},\text{integer}})\\&+n\Big[\sum_{k=0}^{\lfloor\frac{n-1}{4}\rfloor}{\frac{n-1}{2}-k\choose k}q^{2k}-\sum_{k=0}^{\lfloor\frac{n-3}{4}\rfloor}{\frac{n-3}{2}-k\choose k}q^{2k+1}\Big]\Bigg\}.\\
m_{2^{\frac{n-1}{2}}1}=&\frac{1}{2n}\Bigg\{\sum_{k=0}^\frac{n-1}{2}\frac{n}{n-k}{n-k\choose k}q^k+2+2\sum_{p=2}^{\frac{n-1}{2}}(1+pq^\frac{n}{p}\delta_{\frac{n}{p},\text{integer}})\\&-n\Big[\sum_{k=0}^{\lfloor\frac{n-1}{4}\rfloor}{\frac{n-1}{2}-k\choose k}q^{2k}-\sum_{k=0}^{\lfloor\frac{n-3}{4}\rfloor}{\frac{n-3}{2}-k\choose k}q^{2k+1}\Big]\Bigg\}.\\
m_{(\frac{n}{h})^h}=&\frac{1}{2n}\Bigg\{2\sum_{k=0}^\frac{n-1}{2}\frac{n}{n-k}{n-k\choose k}q^k+4cos\frac{2h\pi}{n}+4\sum_{p=2}^{\frac{n-1}{2}}(1+pq^\frac{n}{p}\delta_{\frac{n}{p},\text{integer}})cos(\frac{2hp\pi}{n})\Bigg\},\\&\text{where}~1\leq h\leq
\frac{n-1}{2},~h|n.
\end{split}\end{equation} 
\begin{Ex}
	Let $n=5$, then
	\begin{displaymath} char_{D_5}[\overline{FK}_{C_5}(5)]=m_{1^5}\chi_{1^5}^{}+ m_{221}\chi_{221}^{}+m_{5}\chi_{5}^{}+m_{5'}\chi_{5'}^{},\end{displaymath} 
	where the $m$-coefficients calculated from Equation \eqref{coef2} are:
	\begin{displaymath}
		m_{1^5}=1+q,~ m_{221}=q,~m_{5}=q+q^2,~ m_{5'}=q+q^2.
	\end{displaymath}
	 Substituting the coefficients into the character, Equation \eqref{char5'}, we have
	\begin{displaymath} char_{D_5}[\overline{FK}_{C_5}(5)](q)=(1+q^2)\chi^{}_{1^5}+ q\chi^{}_{221}+(q+q^2)\chi^{}_{5}+(q+q^2)\chi^{}_{5'}.\end{displaymath}Sorting the above in terms of $q$ yields the decomposition in usual degree:
	\begin{displaymath} char_{D_5}[\overline{FK}_{C_5}(5)](q)=\chi^{}_{1^5}+ (\chi^{}_{221}+\chi^{}_{5}+\chi^{}_{5'})q+(\chi^{}_{1^5}+\chi^{}_{5}+\chi^{}_{5'})q^2,\end{displaymath} which upon putting $q=1$ reduces to 
	\begin{displaymath}
	char_{D_5}\overline{FK}_{C_5}(5)=2\chi_{1^5}^{}+\chi_{221}^{}+2\chi_{5'}^{}+2\chi_{5}^{}.
	\end{displaymath} 
\end{Ex}
\subsubsection{Representation decomposition of $D_n$ by conjugation class}
Since the basis of $\overline{FK}_{C_n}(n)$ consists of monomials in variables with no common index (alternatively, matchings in an $n$-cycle), the $S_n$-degree of an element of the basis is product of disjoint $2$-cycles and $1$-cycles. So it belongs to the $S_n$-conjugacy class denoted by $t_2^kt_1^{n-2k}$ where $t_2$ and $t_1$ stand for $2$-cycle and $1$-cycle respectively and where $k$ is the degree of monomial. Since we have the same partition of the basis invariant under $S_n$-conjugacy class as under usual degree, we have essentially the same decomposition as by usual degree.
\subsubsection{Representation decomposition of $D_n$ by set partition type} 
Since no two variable appearing in a monomial $M\in B[\overline{FK}_{C_n}(n)]$ have common indexes, in each part of the set partition there is at most $2$ indexes. So the set partition type of degree $k$ monomial $M$ is of the form
\begin{displaymath}
 \alpha=(\underbrace{2, 2, \cdots, 2}_k,\underbrace{1, 1, \cdots, 1}_{n-2k}),
  ~\text{denoted by}~ 2^k1^{n-2k},
  \end{displaymath} where each $2$ in $\alpha$ represent the two indexes of a variable if the variable appears in $M$, while each $1$ represents the index appearing in no variable in $M$.  Since we have the same partition of the basis invariant under set partition type as under usual degree, we have essentially the same decomposition as by usual degree.
\section*{Conclusion}
We introduced a quotient of Fomin-Kirilov algebra $FK(n)$ denoted by $\overline{FK}_{C_n}(n)$ associated with the $n$-cycle subgraph of the complete graph on $n$ vertexes. We found a nice connection between $FK(n)$ and the theory of $q$-Fibonacci/$q$-Lucas polynomials via this quotient; the Hilbert series of this quotient algebra is nothing but the $q$-Lucas polynomial and its dimension equals the Lucas number $L_n$. We analyzed why $D_n$ representation decomposition of this quotient algebra for usual degree would be essentially the same as for $S_n$ degree and set partition degree. 
 \section*{Acknowledgment} I offer my sincere gratitude to Professor Bergeron for reviewing this paper and for his valuable hints on this paper. 

\end{raggedright}

\begin{thebibliography}{99}
	\bibitem{CHRISTOPH} 
	Christoph B\"arligea. 
	\emph{On the dimension of the Fomin-Kirilov Algebra and
		related algebras}, arXiv:2001.04597v1 [math.QA] 14 Jan 2020.
	\bibitem{bazlov'} Yuri Bazlov. \emph{Nichols-Woronowicz algebra model for Scubert calculus on Coxeter group}.
	
	J. Algebra 297(2), 372-399 (2006). DOI 10.1016/j.jalgebra.2006.01.037.
	
	URL http://dx.doi.org/10.1016/j.jalgebra.2006.01.037.
	\bibitem{Karola}
	Jonan Blasiak, Ricky Ini Liu, and Karola Meszaros.
	\emph{Subalgebras of the Fomin-Kirillov Algebra}a, Preprint arXiv:1310.4112 (2013), 38 pp.
		\bibitem{carlitz}
	L Carlitz. "Fibonacci Notes, I. Zero-one Sequences and Fibonacci Numbers of Higher Order," The Fibonacci
	Quarterly, Vol. 12, No. 1 (February, 1974), pp. 1-10.
	\bibitem{cigler}
	J. Cigler. 
	\emph{Some beautiful $q$-analogues
		of Fibonacci and Lucas polynomials},
	arXiv:1104.2699 [math.HO]
	\bibitem{church} Thomas Churcha, Benson Farbb. 
	\emph{Representation theory and homological stability}, Advances in Mathematics, Volume 245, 1 October 2013, Pages 250-314.	
	\bibitem{fomin}
	Sergey Fomin and Anatol N. Kirillov. \emph{Quadratic algebras, Dunkl elements, and Schubert calculus}.
	Advances in geometry, volume 172 of Progr. Math., pages 147-182. Birkh¨auser Boston, Boston, MA,
	1999.
		\bibitem{fomin''}
S.	Fomin, C. Procesi. 
	\emph{Fibered quadratic Hopf algebras related to Schubert calculus}
	
	J. Algebra 230(1),
	174-183 (2000). DOI 10.1006/jabr.1999.7957.
	
	URL http://dx.doi.org/10.1006/jabr.1999.7957
	\bibitem{kirillov}  Anatol N. Kirillov and Toshiaki Maeno. Noncommutative algebras related with Schubert calculus on
	Coxeter groups. European J. Combin., 25(8):1301-1325, 2004.
	\bibitem{kirillov2} Anatol N. Kirillov. On some algebraic and combinatorial properties of Dunkl elements. Internat. J. Modern Phys. B, 26(27-28):1243012, 28, 2012.
		\bibitem{liu} Ricki Ini Liu. \emph{On the Commutative Quotient of Fomin-Kirillov Algebra}, European Journal of Combinatorics, Volume 54, May 2016, pages 65-75
			\bibitem{mora}
		Teo Mora. 
		\emph{An introduction to commutative
			and noncommutative Grobner
			bases}, Theoretical Computer Science 134 (1994) 131-173
		Elsevier
		131 
		\bibitem{sagan}Bruce E. Sagan.
	\emph{The Symmetric Group, Representations, Combinatorial Algorithms, and Symmetric Functions}, Second Edition, Springer-Verlag New York, Inc., 2001.
\end{thebibliography}
\end{document}